\documentclass[10pt]{article}

\usepackage[utf8]{inputenc}
\usepackage[T1]{fontenc}
\usepackage{amsmath,amssymb,dsfont}
\numberwithin{equation}{section}
\usepackage{microtype}
\usepackage{graphicx,tikz,pgfplots}
\graphicspath{{Images/}}
\pgfplotsset{compat=newest}
\usepackage[hyperref,amsmath,thmmarks]{ntheorem}
\usepackage{aliascnt}
\usepackage[a4paper,centering,bindingoffset=0cm,marginpar=2cm,margin=2.5cm]{geometry}
\usepackage[pagestyles]{titlesec}
\usepackage[font=footnotesize,format=plain,labelfont=sc,textfont=sl,width=0.75\textwidth,labelsep=period]{caption}
\usepackage{url}

\usepackage[backend=biber,maxnames=10,backref=true,hyperref=true,safeinputenc]{biblatex}
 
\DefineBibliographyStrings{english}{%
	backrefpage = {cited on page},
	backrefpages = {cited on pages},
}

\title{Diffusion Tensor Regularization with Metric Double Integrals}
\author{Leon Frischauf$^1$\\{\footnotesize\href{mailto:leon.frischauf@univie.ac.at}{leon.frischauf@univie.ac.at}}
\and Melanie Melching$^1$\\{\footnotesize\href{mailto:melanie.melching@univie.ac.at}{melanie.melching@univie.ac.at}}
\and Otmar Scherzer$^{1,2}$\\{\footnotesize\href{mailto:otmar.scherzer@univie.ac.at}{otmar.scherzer@univie.ac.at}}}
\date{}

\DeclareFieldFormat[report]{title}{``#1''}
\DeclareFieldFormat[book]{title}{``#1''}
\AtEveryBibitem{\clearfield{url}}
\AtEveryBibitem{\clearfield{note}}

\titleformat{\section}[block]{\large\sc\filcenter}{\thesection.}{0.5ex}{}[]
\titleformat{\subsection}[runin]{\bf}{\thesubsection.}{0.5ex}{}[.]

\usepackage[pdftex,colorlinks=true,linkcolor=teal,citecolor=green,urlcolor=blue,bookmarks=true,bookmarksnumbered=true]{hyperref}
\hypersetup
{
    pdfauthor={L. Frischauf, M. Melching, O. Scherzer},
    pdfsubject={Subject},
    pdftitle={Diffusion Tensor Regularization with Matric Double Integrals},
    pdfkeywords={regularization, diffusion tensor imaging, metric, double integral, fractional Sobolev space, medical imaging}
}

\newpagestyle{headers}
{%
	\headrule
	\sethead%
	[\footnotesize\thepage]%
	[\footnotesize\sc L. Frischauf, M. Melching, O. Scherzer]%
	[]%
	{}%
	{\footnotesize\sc Diffusion Tensor Regularization}%
	{\footnotesize\thepage}
	\setfoot{}{}{}
}
\newtheorem{lemma}{Lemma}[section]

\newaliascnt{proposition}{lemma}

\aliascntresetthe{proposition}

\newaliascnt{corollary}{lemma}
\newtheorem{corollary}[corollary]{Corollary}
\aliascntresetthe{corollary}

\newaliascnt{theorem}{lemma}
\newtheorem{theorem}[theorem]{Theorem}
\aliascntresetthe{theorem}

\theorembodyfont{\normalfont}
\newaliascnt{definition}{lemma}
\newtheorem{definition}[definition]{Definition}
\aliascntresetthe{definition}

\newaliascnt{assumption}{lemma}
\newtheorem{assumption}[assumption]{Assumption}
\aliascntresetthe{assumption}

\newaliascnt{remark}{lemma}
\newtheorem{remark}[remark]{Remark}
\aliascntresetthe{remark}

\theoremstyle{nonumberplain}
\theoremseparator{:}
\theoremheaderfont{\normalfont\itshape}

\theoremsymbol{\ensuremath{\square}}
\newtheorem{proof}{Proof}

\newcommand{\N}{\mathds{N}}

\newcommand{\R}{\mathds{R}}

\let\RE\Re
\let\Re=\undefined
\DeclareMathOperator{\Re}{\RE e}
\let\IM\Im
\let\Im=\undefined
\DeclareMathOperator{\Im}{\IM m}

\newcommand{\abs}[1]{\left|#1\right|}
\newcommand{\norm}[1]{\left\|#1\right\|}
\newcommand{\set}[1]{\left\{#1\right\}}

\newcommand{\e}{\mathrm e}
\let\ii\i
\renewcommand{\i}{\mathrm i}
\renewcommand{\d}{d}
\usepackage{mathtools}
\usepackage{bbold}
\usepackage{pifont}
\usepackage[]{algorithm2e}
\usepackage[labelformat=simple]{subcaption}

\newcommand{\ve}{\varepsilon}
\newcommand{\rarr}{\rightarrow}
\newcommand{\vd}{w^\delta}
\newcommand{\defeq}{\vcentcolon=}
\newcommand{\eqdef}{=\vcentcolon}
\newcommand{\enorm}[1]{\left|#1\right|}
\newcommand{\dx}{\mathrm{d}x}
\newcommand{\dxy}{\mathrm{d}(x,y)}
\DeclareMathOperator*{\argmin}{arg\,min}

\newcommand{\Log}{\mathrm{Log}}
\newcommand{\Exp}{\mathrm{Exp}}
\newcommand{\SPD}{\mathrm{SPD}}
\newcommand{\SPDs}{\SPD^{\mathrm{spec}}}
\newcommand{\SPDl}{\SPD^{\mathrm{Log}}}
\newcommand{\diag}{\mathrm{diag}}
\newcommand{\SYM}{\mathrm{SYM}}
\newcommand{\dRM}{\d_{\R^{m \times m}}}
\newcommand{\F}{\mathcal{F}}
\newcommand{\Ft}{\tilde{\mathcal{F}}}
\newcommand{\Reg}{\Phi}

\newcommand{\Proj}{\mathrm{P}}

\begin{document}

\maketitle
\thispagestyle{empty}
\begin{center}
\hspace*{5em}
\parbox[t]{12em}{\footnotesize
\hspace*{-1ex}$^1$Faculty of Mathematics\\
University of Vienna\\
Oskar-Morgenstern-Platz 1\\
A-1090 Vienna, Austria}
\hfil
\parbox[t]{17em}{\footnotesize
\hspace*{-1ex}$^2$Johann Radon Institute for Computational\\
\hspace*{1em}and Applied Mathematics (RICAM)\\
Altenbergerstraße 69\\
A-4040 Linz, Austria}
\end{center}


\begin{abstract}
In this paper we propose a variational regularization method for denoising and inpainting of diffusion tensor magnetic resonance images. We consider these images as manifold-valued Sobolev functions, i.e. in an infinite dimensional setting, which are defined appropriately. The regularization functionals are defined as double integrals, which are equivalent to Sobolev semi-norms in the Euclidean setting. We extend the analysis of \fullcite{CiaMelSch19} concerning stability and convergence of the variational regularization methods by a uniqueness result, apply them to diffusion tensor processing, and validate our model in numerical examples with synthetic and real data. 
\end{abstract}

\section{Introduction}
\label{sec:Intro}
In this paper we investigate denoising and inpainting of \emph{diffusion tensor (magnetic resonance) images} (DTMRI) with 
a derivative-free, non-local variational regularization technique proposed, implemented and 
analyzed first in \cite{CiaMelSch19}.

The proposed regularization functionals generalize equivalent definitions of the Sobolev semi-norms, which have been derived in the fundamental work of \cite{BouBreMir01} 
and follow up work \cite{Dav02,Pon04b}. These papers provide a derivative-free representation of Sobolev semi-norms for intensity and vector-valued functions. The beauty of this representation is that it allows for a straight forward definition of Sobolev energies of \emph{manifold-valued data} (see \cite{CiaMelSch19}), without requiring difficult differential geometric concepts (as 
for instance in \cite{GiaMuc06,Con17}).

Diffusion tensor images are considered to be re-presentable 
as functions from an image domain $\Omega \subset \R^n$, with $n=2,3,$ respectively, into the manifold of 
symmetric, positive definite matrices in $\R^{m \times m}$, denoted by $K$ in the following - 
for DTMRI images $m=3$. 
Therefore, they are ideal objects to check the efficiency of the proposed regularization techniques. 
A measured diffusion tensor is often very noisy and post-processing steps for noise removal are important. Even more, due to the noise it is possible that a measured tensor has negative eigenvalues, which is not physical, and thus often the whole tensor at this point is omitted, leading to incomplete data. Then the missing information has to be inpainted before visualization.

Variational regularization of vector- and matrix- manifold-valued functions has been 
considered before, for instance in \cite{WeiBro02,TscDer02,TscDer05,BacBerSteWei16,BerChaHiePerSte16,LelStrKoeCre13,WeiDemSto14, OstWan20} and \cite{BerFitPerSte18, BreHolStoWei18, Pen19}. Non-local regularization formulations were studied for example in \cite{BerTen17,LauNikPerSte17} for filtering tensor-valued functions, and also \cite{GilOsh08,LelPapSchSpe15} for filtering intensity images.
An overview of diffusion and regularization techniques for vector-, and matrix-valued data is given in \cite{WeiBro02}. 

Variational methods for denoising and inpainting attempt to find a good compromise between 
matching some given noisy, tensor-valued data $\vd:\Omega \rarr K$ and prior 
information on the desired solution $w^0:\Omega \rarr K$, also called noise free or ideal solution.
The choice of prior knowledge on $w^0$ is that 
\begin{enumerate}
	\item it is an element of the \emph{set} $W^{s,p}(\Omega; K)$, with a metric $d$ on $K$, the set of positive definite matrices, which is a subset of the fractional Sobolev space    
	 $W^{s,p}(\Omega; \R^{m \times m})$, with $s \in (0,1)$ and $p \in (1, \infty)$, and
	\item that 
	 \begin{equation} \label{eq:phi}\boxed{
	   \Reg^l_{[\d]}(w) \defeq 
				\int_{\Omega\times \Omega} \frac{\d^p(w(x), w(y))}{\enorm{x-y}^{n+p s}} \rho^l(x-y) \dxy}
	\end{equation} 
    is relatively small. The function $\rho$ is a non-negative and radially symmetric mollifier with an on-off indicator $l \in \{0,1\}$ denoting 
    whether the mollifier is used or not. Note, that in case that $\d = \dRM$ is the Euclidean metric and if we choose in addition $l=0$,  
    $\Reg_{[\dRM]}^0 $ becomes the fractional Sobolev semi-norm. 
\end{enumerate} 
The compromise of approximating $\vd$ with a function in $W^{s,p}(\Omega; K)$ with a small energy term $\Reg^l_{[\d]}(w)$ is achieved by minimization of the functional
\begin{equation} \label{eq:reg}\boxed{
		\begin{aligned} 
			\F_{[\d]}^{\alpha,\vd}(w) \defeq
			& \int_\Omega \chi_{\Omega \setminus D}(x) \d^p(w(x),\vd(x)) \dx + 
			\alpha \Reg_{[\d]}^l (w),
	\end{aligned}}
\end{equation}

where the parameter $\alpha > 0$ determines the preference of staying close to the given function $\vd$ in 
$\Omega \setminus D$ and a small energy $\Reg^l_{[\d]}(w)$. One should not confuse the energy term 
with double integral representations approximating semi-norms on manifolds (see for instance \cite{Heb96,KreMor19,EffNeuRum20}).

The indicator function of $\Omega \backslash D$, 
\begin{equation*}
\chi_{\Omega \setminus D}(x) = \begin{cases}
1 &\mbox{ if } x \in \Omega \setminus D,\\
0 &\mbox{ otherwise},
\end{cases}
\end{equation*}

used in \autoref{eq:reg} allows us to consider the two tasks of denoising ($D = \emptyset$) and inpainting ($D \neq \emptyset$) within one analysis. 

The paper is organised as follows: In \autoref{sec:Not} we constitute our notation and setting used to analyze variational methods for DTMRI data processing. 
We review regularization results from \cite{CiaMelSch19} in \autoref{sec:Regularization}. In \autoref{sec:DT} we verify that these results from \autoref{sec:Regularization} are applicable in the context of diffusion tensor imaging, meaning that we show that the functional $\F_{[\d]}^{\alpha,\vd}$ defined in \autoref{eq:reg} attains a minimizer and fulfills a stability as well as a convergence result. Furthermore we extend the analysis and give a uniqueness result using differential geometric properties of symmetric, positive definite matrices, where it is of particular importance, that these matrices endowed with the log-Euclidean metric form a flat Hadamard manifold. 
In \autoref{sec:Numerics} we give more details on the numerical minimization of the regularization functional, and discuss different variants.
In the last \autoref{sec:experiments} we show numerical results for denoising and inpainting problems of synthetic and real DTMRI data.
%

\section{Notation and Setting}
\label{sec:Not}
In the beginning we summarize basic notation and assumptions used throughout the paper.
In the theoretical part we work with general dimensions $n,m \in \N$ while we consider the 
particular case $n=2,m=3$, that is 2-dimensional slices of a 3-dimensional DTMRI image, in the numerical examples in \autoref{sec:experiments}.
\begin{assumption}\label{ass:metric}
	\begin{enumerate}
		\item $\Omega \subset \R^n$ is a  nonempty, bounded and 
		connected open set with Lipschitz boundary and $D \subset \Omega$ is measurable.
		\item $p \in (1, \infty), \ s \in (0,1)$ and $l \in \{0,1\}$.
		\item \label{itm:K} $K \subseteq \R^{m \times m}$ is a nonempty and closed subset of $\R^{m \times m}$.
		\item $\dRM: \R^{m \times m} \times \R^{m \times m} \rarr [0, \infty)$ denotes the Euclidean distance induced by the (Frobenius norm) on $\R^{m \times m}$ and 
		\item $\d \defeq \mathrm{d}_{K}: K \times K \rarr [0, \infty)$ denotes an arbitrary metric on $K$ which is equivalent to $\dRM$. \label{itm:EquivMetric}
	\end{enumerate}
\end{assumption}
Moreover, we need the definition of a mollifier which appears in the regularizer of the functional in \autoref{eq:reg}.
\begin{definition}[Mollifier]\label{def:mol} 
	We call $\rho \in C_c^\infty(\R^n;\R)$ a \emph{mollifier} if
	\begin{itemize}
		\item $\rho$ is a non-negative, radially symmetric function,
		\item $\int_{\R^n} \rho (x) \dx = 1$ and 
		\item there exists some $0 < \tau < \norm{\rho}_{L^\infty(\R^n;\R)}$ 
		and $\eta \defeq \eta_\tau > 0$ such that
		$\set{ z \in \R^n :\rho(z) \geq \tau } = \set{z \in \R^n : \enorm{z} \leq \eta }$.
	\end{itemize}
\end{definition}
The last condition holds for instance if $\rho$ is radially decreasing satisfying $\rho(0) > 0$.

\subsection{Diffusion tensors}
\label{ssec:DT}
	It is commonly assumed 
	that the recorded diffusion tensor images are functions with values which are symmetric, positive definite matrices.
   Hence we make the assumption that
   \begin{equation*}
   		w,\vd: \Omega \rarr \SPD(m),
   	\end{equation*}
   	where $\SPD(m)$ is the set of symmetric, positive definite, real $m \times m$ matrices defined below in \autoref{eq:defSPD}. When working with data from MRI measurements $m=3$. 
   	
   	In the following definition we summarize sets of matrices and associated norms on the sets:

\begin{definition}\label{def:sets}
\begin{itemize}
    \item The vector space of \emph{symmetric matrices}
    	\begin{equation}\label{eq:defSYM}
    	\SYM(m)
    	\defeq \big\{ M \in \R^{m \times m}  :  M^T = M \big\}.
    	\end{equation}
    \item Additionally, we define set of \emph{symmetric, positive definite $m \times m$ matrices}
    	\begin{equation}\label{eq:defSPD}
    	\SPD(m)
    	\defeq \big\{ M \in \SYM(m)  :  x^TMx>0 \text{ for } x\in \R^m\setminus\{0\} \big\}.
    	\end{equation}
	\item The set of \emph{symmetric, positive definite matrices with bounded spectrum}
	\begin{equation}\label{eq:SPDs}
	\SPDs_{[\underline{\ve}, \bar{\ve}]}(m)
	\defeq \big\{ M \in \SPD(m)  :  \mathrm{spec}(M) \subseteq [\underline{\ve}, \bar{\ve}] \big\},
	\end{equation}
	where $\mathrm{spec}$ denotes the spectrum of a given matrix. For diffusion tensors the spectrum is real.
	\item 
	The set of \emph{symmetric, positive definite matrices with bounded logarithm}
	\begin{equation}\label{eq:SPDb}
	\SPDl_z(m)
	\defeq \big\{ M \in \SPD(m)  :  \norm{\Log(M)}_F \leq z \big\},
	\end{equation}
	where $\Log$ is the matrix logarithm defined later in \autoref{def:MatrixProp} \autoref{itm:Log} and $\norm{\cdot}_F$ denotes the \emph{Frobenius norm} defined as
	\begin{equation}\label{eq:FrobNorm}
	\|M\|_F = \sqrt{\sum_{i,j=1}^{m} |m_{ij}|^2}.
	\end{equation}
\end{itemize}
\end{definition}

When working with DTMRI data, in particular in \autoref{sec:experiments}, we will chose $K=\SPDl_z(3)$. In the general theory stated in the next \autoref{sec:Regularization} \emph{any} nonempty and bounded set can be taken.

From now on and whenever possible we omit the space dimension and write $\SYM,\SPD, \SPDs_{[\underline{\ve}, \bar{\ve}]}$ and $\SPDl_z$ instead of $\SYM(m),\SPD(m), \SPDs_{[\underline{\ve}, \bar{\ve}]}(m)$ and $\SPDl_z(m)$.

\subsection{Fractional Sobolev spaces}
\label{ssec:SobolevSpaces}
Moreover, we need the definition of fractional Sobolev spaces and associated subsets.
\begin{definition}[Sobolev spaces of fractional order]\label{def:Spaces} Let \autoref{ass:metric} hold.
	\begin{itemize}
		\item We denote by $L^p(\Omega;\R^{m \times m})$ the \emph{Lebesgue space} of matrix-valued functions.
		\item The \emph{Sobolev space} $W^{1,p}(\Omega;\R^{m \times m})$ consists of all weakly differentiable functions in $L^p(\Omega;\R^{m \times m})$ for which
		\begin{equation*}
			\norm{w}_{W^{1,p}(\Omega;\R^{m \times m})} 
			\defeq \left( \norm{w}_{L^p(\Omega;\R^{m \times m})}^p 
			+  \int_\Omega  \norm{\nabla w(x)}^p_F \dx 
			\right)^{1/p} < \infty\;,
		\end{equation*}
		where $\nabla w$ is the Jacobian of $w$ and $\abs{w}_{W^{1,p}(\Omega;\R^{m \times m})} \defeq \left( \int_\Omega \norm{\nabla w(x)}_F^p \dx \right)^{1/p}$ is the Sobolev semi-norm.
		\item The \emph{fractional Sobolev space} of order $s$ is defined (cf. \cite{Ada75}) as the set
		\begin{gather*}
			W^{s,p}(\Omega;\R^{m \times m}) \defeq  \set{ w \in L^p(\Omega;\R^{m \times m}) :\frac{\norm{w(x)-w(y)}_F}{\enorm{x-y}^{\frac{n}{p}+s}} \in L^p (\Omega \times \Omega;\R) } 
		\end{gather*}
		equipped with the norm
		\begin{equation}\label{eq:s_norm}
			\norm{w}_{W^{s,p}(\Omega;\R^{m \times m})} \defeq \left(\norm{w}_{L^p(\Omega;\R^{m \times m})}^p + \abs{w}_{W^{s,p}(\Omega;\R^{m \times m})}^p \right)^{1/p},
		\end{equation}
		where $\abs{w}_{W^{s,p}(\Omega;\R^{m \times m})}$ is the semi-norm on $W^{s,p}(\Omega;\R^{m \times m})$, defined by
		\begin{equation}\label{eq:s_seminorm}
			\abs{w}_{W^{s,p}(\Omega;\R^{m \times m})} \defeq  
			\left(\int_{\Omega\times \Omega} \frac{\norm{w(x)-w(y)}_F^p}{\enorm{x-y}^{n+ps}} \dxy \right)^{1/p}
			\quad \text{ for all } w \in W^{s,p}(\Omega;\R^{m \times m}).
		\end{equation}
		\item We define the \emph{fractional Sobolev set} of order $s$ with data in $K$ as
		\begin{equation} \label{eq:sk}
			W^{s,p}(\Omega;K) \defeq \set{ w \in W^{s,p}(\Omega;\R^{m \times m}) : w(x) \in K \text{ for a.e. } x \in \Omega }.
		\end{equation}
		The \emph{Lebesgue set} with data in $K$ is defined as
		\begin{equation} \label{eq:lk}
			L^p(\Omega;K)  \defeq \set{ w \in L^p(\Omega;\R^{m \times m}) : w(x) \in K \text{ for a.e. } x \in \Omega}.
		\end{equation}
	\end{itemize}
\end{definition}
Note that $L^p(\Omega;K)$ and $W^{s,p}(\Omega;K)$ are sets and not linear spaces because summation of elements in $K$ is typically not closed in $K$. 

\section{Metric double integral regularization on closed subsets}
\label{sec:Regularization}
We start this section by stating conditions under which the regularization functional in \autoref{eq:reg} attains a minimizer and fulfills a stability as well as a convergence result. 
Therefore we recall results established in \cite{CiaMelSch19}. There the authors define a regularization functional 
inspired by the work of \citeauthor{BouBreMir01} \cite{BouBreMir01,Pon04b, Dav02}. The analysis in turn is based on \cite{SchGraGroHalLen09}. We apply these results to diffusion tensor image denoising and inpainting in the next section.

We start by stating general conditions on the exact data $w^0$, the noisy data $\vd$ and the functional $\F_{[\d]}^{\alpha,\vd}$, defined in \autoref{eq:reg}. 
\begin{assumption} \label{ass:gc} 
	Let \autoref{ass:metric} hold. 
	Moreover, let $w^0,\vd \in L^p(\Omega;K)$ and let $\rho$ be a mollifier as defined in \autoref{def:mol}. 
	We assume that 
	\begin{enumerate}
		\item	 
		For every $t > 0$ and $\alpha > 0$ the level sets
		\begin{equation*}
			\text{level}(\F^{\alpha,w^0}_{[\d]};t) \defeq \set{  w \in W^{s,p}(\Omega;K):\  \F^{\alpha,w^0}_{[\d]}(w) \leq t  }  
		\end{equation*}
		are weakly sequentially pre-compact in $W^{s,p}(\Omega;\R^{m \times m})$. \label{itm:F2}
		\item	  
		There exists $\bar{t} > 0$ such that $\text{level}(\F^{\alpha,w^0}_{[\d]};\bar{t})$ is nonempty. \label{itm:F3}
	\end{enumerate}
\end{assumption}

\begin{remark}\label{rem:AssFulfilled}
If \autoref{ass:metric} is fulfilled and
in particular when performing image denoising ($D = \emptyset$) or inpainting ($D \neq \emptyset$) of functions with values 
in $K$, then the functional \autoref{eq:reg} with $\Reg_{[\d]}^l$ as in \autoref{eq:phi} defined on $W^{s,p}(\Omega;K)$ satisfies 
\autoref{ass:gc} (cf. \cite{CiaMelSch19}).
\end{remark}
According to \cite{CiaMelSch19} we now have the following result giving existence of a minimizer of the functional in \autoref{eq:reg} as well as a stability and convergence result:
\begin{theorem}\label{thm:ex} 
	Let \autoref{ass:gc} hold (which is guaranteed by \autoref{rem:AssFulfilled}). 
	For the functional defined in \autoref{eq:reg} over $W^{s,p}(\Omega;K)$ with $\Reg_{[\d]}^l$ defined in \autoref{eq:phi}
	the following results hold:
	\begin{description} 	 
		\item{\emph{Existence:}} For every $v \in L^p(\Omega;K)$ and $\alpha > 0$ the functional 
		$\F_{[\d]}^{\alpha,v}: W^{s,p}(\Omega;K) \rightarrow [0, \infty)$ attains a minimizer in $W^{s,p}(\Omega;K)$. 
		\item{\emph{Stability:}} Let $\alpha > 0$ be fixed, $\vd \in L^p(\Omega;K)$ and let $(v_k)_{k \in \N}$ be a sequence 
		in $L^p(\Omega;K)$ such that $\norm{\vd-v_k}_{L^p(\Omega;K)} \rightarrow 0$. Then every sequence 
		$(w_k)_{k \in \N}$ satisfying 
		\begin{equation*}
			w_k \in \arg \min \set{ \F_{[\d]}^{\alpha,v_k}(w):\ w \in W^{s,p}(\Omega;K) }
		\end{equation*}	
		has a converging subsequence with respect to the weak topology of $W^{s,p}(\Omega;\R^{m \times m})$. 
		The limit $\tilde{w}$ of every such converging subsequence $(w_{k_j})_{j \in \N}$ is a minimizer of 
		$\F_{[\d]}^{\alpha,\vd}$. Moreover, $(\Reg_{[\d]}^l(w_{k_j}))_{j \in \N}$ converges to $\Reg_{[\d]}^l(\tilde{w})$.	
		
		\item{\emph{Convergence:}} Let $\alpha:(0, \infty) \rightarrow (0,\infty)$ be a function satisfying $\alpha(\delta) \rightarrow 0$ and 
		$\frac{\delta^p}{\alpha(\delta)} \rightarrow 0$ for $\delta \to 0$.
		
		Let $(\delta_k)_{k \in \N}$ be a sequence of positive real numbers converging to $0$. Moreover, let 
		$(v_k)_{k \in \N}$ be a sequence in $L^p(\Omega;K)$ with $\norm{w^0-v_k}_{L^p(\Omega;K)} \leq \delta_k$ and 
		set $\alpha_k \defeq \alpha(\delta_k)$. Then every sequence  $(w_k)_{k \in \N}$ defined as
		$$w_k \in \arg \min \set{ \F^{\alpha_k,v_k}_{[\d]}(w):\ w \in W^{s,p}(\Omega;K)} $$ 
		has a \emph{weakly converging} subsequence 
		$w_{k_j} \rightharpoonup w^0$ as $j \to \infty$ (with respect to the topology of $W^{s,p}(\Omega;\R^{m \times m})$).
		In addition, $\Reg_{[\d]}^l(w_{k_j}) \rightarrow \Reg_{[\d]}^l(w^0)$. 
		Moreover, it follows that even
		$w_k \rightharpoonup w^0$ weakly (with respect to the topology of $W^{s,p}(\Omega;\R^{m \times m})$) and 
		$\Reg_{[\d]}^l(w_{k}) \rightarrow \Reg_{[\d]}^l(w^0)$.
	\end{description}
\end{theorem}
In the theorem above, stability (with respect to the $L^p$-norm) ensures that the minimizers of $\F_{[\d]}^{\alpha,\vd}$ depend continuously on the given data $\vd$. 
We emphasize that in an Euclidean setting (that is on $W^{s,p}(\Omega,\R^{m \times m})$, for $s > 0$ and $p > 1$, one could expect convergence in even stronger norms. However, here, we have to make sure that the traces into $K \subseteq \R^{m \times m}$ are well-defined in appropriate Sobolev spaces, which requires additional compactness assumptions, or in other words, stronger regularization.

In the next section we apply \autoref{thm:ex} to diffusion tensor images, i.e. when choosing $K$ as a closed subset of the symmetric, positive definite matrices.

\section{Diffusion tensor regularization}
\label{sec:DT}
The goal of this section is to define appropriate fractional order Sobolev sets as defined in \autoref{eq:sk} of functions which can represent diffusion tensor images. To this end we use 
the set of symmetric, positive definite $m\times m$ matrices with bounded logarithm (defined in \autoref{eq:SPDb})
	\begin{equation}\label{eq:K}
	\boxed{
		K = \SPDl_z
		}
	\end{equation}
and associate it with the \emph{log-Euclidean metric}, 
defined below in \autoref{eq:LEmetric}. This metric was shown to be an adequate distance measure for DTMRI, see e.g. \cite{FilArsPenAya07,ArsFilPenAya05}.

Below we show that \autoref{thm:ex} applies to the regularization functional in \autoref{eq:reg} with the particular choice $K = \SPDl_z$. In addition to what follows from the general theory 
from \cite{CiaMelSch19} in a straight forward manner we present a uniqueness result for the 
minimizer of the regularization functional.

We begin by defining needed concepts from matrix calculus. When working with symmetric, positive definite matrices many of the operations below reduce to their scalar counterpart applied to the eigenvalues.

\subsection{Matrix calculus}
We start this section by repeating basic definitions known from matrix calculus (see for instance \cite{Oss92}). Especially the matrix logarithm is needed to define the log-Euclidean metric on the symmetric, positive definite matrices. 
\begin{lemma}[Matrix properties] \label{rem:SPDProp}
	\begin{enumerate}
		\item \emph{Eigendecomposition:} \label{itm:Decomp} 
		Let $A \in \SYM$ with eigenvalues $(\lambda_i)_{i=1}^m$. Then we can write
		\begin{equation*}
		A = U\Lambda U^T,
		\end{equation*}
		where $U \in \R^{m \times m}$ is the orthonormal matrix whose i-th column consists of the i-th normalized eigenvector of $A$. Hence we have that $U U^T  = \mathbb{1}_m$, where $\mathbb{1}_m$ denotes the identity matrix in $\R^{m \times m}$. $\Lambda$ is the diagonal matrix whose diagonal entries are the corresponding eigenvalues, $\Lambda = \diag(\lambda_1, \dots \lambda_m)$. 
		\item \label{itm:unitary} If $U,V \in \R^{m \times m}$ are both unitary then so are $UV^T, U^TV, VU^T$ and $V^TU$. 
	\end{enumerate}
\end{lemma}

Next we state
the definitions of the matrix exponential and logarithm, see in particular \cite{PenFilAya06,FilArsPenAya07a}.

\begin{definition}\label{def:MatrixProp}
	Let $A,B \in \SYM$ with corresponding eigendecompositions $A = U\Lambda_A U^T$ and $B = V \Lambda_B V^T$, where $U,V \in R^{m \times m}$ unitary and 
	$\Lambda_A = \diag(\lambda_1, \dots, \lambda_m), \Lambda_B = \diag(\mu_1, \dots, \mu_m) \in \R^{m \times m}$ diagonal.
	\begin{enumerate}
	\item \emph{Exponential map:} The exponential map is defined as
		\begin{equation*}
			\Exp(A) = \Exp(U\Lambda_A U^T) = U \Exp(\Lambda_A) U^T.
		\end{equation*}
		It holds that
		\begin{equation*}
			\Exp(\Lambda_A) = \diag(\e^{\lambda_1}, \dots, \e^{\lambda_m}),
		\end{equation*}
		where $\e: \R \rarr \R_{\geq 0}$ denotes the (scalar) exponential function. 
	   $\Exp: \SYM \rarr \SPD$
		is a diffeomorphism \cite[Thm. 2.8]{FilArsPenAya07a}.
	\item \label{itm:Log}\emph{Logarithm: } If $\Exp(A) = B$, then A is the matrix logarithm of $B$. It is defined as
		\begin{equation*}
			\Log(B) = \Log(V \Lambda_B V^T) = V \Log(\Lambda_B)V^T.
		\end{equation*} 
		Moreover,
		\begin{equation*}
			\Log(\Lambda_B) = \diag(\log(\mu_1), \dots, \log(\mu_m)),
		\end{equation*}
		where $\log: \R_{\geq 0} \rarr \R$ is the (scalar) natural logarithm, i.e. $\log \defeq \log_\e$. \\
		 When restricting to symmetric, positive definite matrices $\Log: \SPD \rarr \SYM$ is a diffeomorphism \cite[Thm. 2.8]{FilArsPenAya07a}.
	\end{enumerate}
\end{definition}

The previous \autoref{def:MatrixProp} shows that the exponential and logarithm of a symmetric (positive definite) matrix can be computed easily due to the eigendecomposition (see \autoref{rem:SPDProp}) by calculating the
scalar exponential map and logarithm of the eigenvalues.
	
\begin{remark}[Matrix logarithm]
	For a general matrix in $\R^{m \times m}$ the matrix logarithm is not unique. Matrices with positive eigenvalues have a \emph{unique, symmetric} logarithm, called the \emph{pricipal-logarithm} \cite{FilArsPenAya07a}.
\end{remark}

The next lemma states properties of the Frobenius norm (recall \autoref{eq:FrobNorm}).
\begin{lemma}[Properties of Frobenius norm] \label{lem:PropFrobNorm}
	\begin{enumerate}
		\item  Let $A,B \in \R^{m \times m}$ be symmetric and skew-symmetric, respectively, i.e. $A = A^T, \ B = -B^T$. Then
		\begin{equation}\label{eq:PropFrob}
		\norm{A+B}_F^2 = \norm{A}_F^2 + \norm{B}_F^2.
		\end{equation}
		\item  The Frobenius norm is unitary invariant, i.e.
		\begin{equation}\label{eq:OrthogonalInv}
		\norm{A}_F = \norm{UAV}_F
		\end{equation}
		for $A \in \R^{m \times m}$ and $U,V \in \R^{m \times m}$ unitary.
		\item  If $A \in \SPD$ with (positive) eigenvalues $(\lambda_i)_{i=1}^m$ then
		\begin{equation}\label{eq:FrobeniusEigenvalues}
		\norm{\Log(A)}_F = \left( \sum_{i=1}^{m} \log^2(\lambda_i) \right)^{1/2}.
		\end{equation}
	\end{enumerate}
\end{lemma}

\begin{proof}
	The proof of the first item is straightforward by using the definition of $\norm{\cdot}_F$ in \autoref{eq:FrobNorm}. The second item follows directly by considering the trace representation of the Frobenius norm \cite{Wer92}:
	\begin{equation*}
	\norm{UAV}_F^2 = \mathrm{trace}\left( (UAV)^T UAV \right) 
	=  \mathrm{trace}\left( (V^TA^T AV \right) 
	= \mathrm{trace}\left( ( AV V^TA^T \right) 
	= \norm{A}_F ^2.
	\end{equation*}
	 The third item is a direct consequence of \autoref{rem:SPDProp} \autoref{itm:Decomp}, \autoref{def:MatrixProp} \autoref{itm:Log} and \autoref{eq:OrthogonalInv}.
\end{proof}

The last lemma of this subsection deals with the set $\SPDs_{[\underline{\ve},\bar{\ve}]}$, the set of symmetric, positive definite matrices with bounded spectrum in the interval $[\underline{\ve},\bar{\ve}]$, defined in \autoref{eq:SPDs}. We need this result later in the numerical implementation for defining a suitable projection: 

Given an arbitrary matrix $A \in \R^{m \times m}$ there always exists a unique matrix $M \in \SPDs_{[\underline{\ve},\bar{\ve}]}$ which is closest in the Frobenius norm, i.e.
\begin{equation}\label{eq: SPDBapprox}
M = \argmin_{X \in \SPDs_{[\underline{\ve},\bar{\ve}]}} \|A-X\|_F^2.
\end{equation} 
The minimizing matrix $M$ can be computed explicitly as stated in the following lemma.  
The proof is done in a similar way as in \cite[Theorem 2.1]{Hig88} and included here for completeness.

\begin{lemma}\label{lem:Approx}
	Let $A \in \R^{m \times m}$. Define $B \defeq \frac{1}{2}(A + A^T)$ and $C \defeq \frac{1}{2}(A - A^T)$ as the symmetric and skew-symmetric parts of $A$, respectively. Let $(\lambda_i)_{i=1}^m$ be the eigenvalues of $B$ which can be decomposed into $B = Z \Lambda Z^T$, where $Z$ is a unitary matrix, i.e. $ZZ^T = Z^TZ = \mathbb{1}_m$, and $\Lambda = \diag(\lambda_1, \dots, \lambda_m)$.
	Then the unique minimizer of 
	\begin{equation}\label{eq:minPb}
		\min_{X \in \SPDs_{[\underline{\ve},\bar{\ve}]}} \|A-X\|_F^2,
	\end{equation}
	where $\SPDs_{[\underline{\ve},\bar{\ve}]}$ is defined in \autoref{eq:SPDs}, is 
	$ZYZ^T$ with $Y = \mathrm{diag}(d_1, \dots, d_m)$, where 
	$d_i \defeq \begin{cases}
	\lambda_i &\mbox{ if } \lambda_i \in [\underline{\ve},\bar{\ve}], \\
	\bar{\ve} &\mbox{ if } \lambda_i > \bar{\ve}, \\
	\underline{\ve} &\mbox{ if } \lambda_i < \underline{\ve}.
	\end{cases}$.
\end{lemma} 

\begin{proof}
	By definition of $B$ and $C$ we can write $A = B+C$ and thus
	\begin{equation*}
	\norm{A-X}_F^2 = \norm{B+C-X}_F^2 = \norm{B-X}_F^2 + \norm{C}_F^2,
	\end{equation*}
	where we used \autoref{eq:PropFrob} in the second equality. The problem in \autoref{eq:minPb} thus reduces to finding
	\begin{equation*}
	\argmin_{X \in \SPDs_{[\underline{\ve},\bar{\ve}]}} \|B-X\|_F^2.
	\end{equation*}
	The matrix $B$ is symmetric and thus we can write $B = Z \Lambda Z^T$, where $Z \in \R^{m \times m}$ is a unitary matrix whose columns are the eigenvectors of $B$ and $\Lambda \in \R^{m \times m}$ is a diagonal matrix whose entries are the eigenvalues of $B$, i.e. $\Lambda = \mathrm{diag}(\lambda_1, \dots, \lambda_m)$. Let $Y = Z^TXZ$ be similar to $X$ so that $\mathrm{spec}(Y) \subset [\underline{\ve},\bar{\ve}]$. Then we obtain by using \autoref{eq:OrthogonalInv}
	\begin{align*}
	\norm{B-X}_F^2 
	&= \norm{\Lambda - Y}_F^2 
	= \sum_{\{i,j : i \neq j\}} y_{ij}^ 2 + \sum_{i=1}^m (\lambda_i-y_{ii})^2 \\
	& = \sum_{\{i,j : i \neq j\}} y_{ij}^ 2 
	+ \sum_{\{i: \lambda_i \in [\underline{\ve},\bar{\ve}]\}} (\lambda_i-y_{ii})^2
	+ \sum_{\{i: \lambda_i > \bar{\ve}\}} (\lambda_i-y_{ii})^2 
	+ \sum_{\{i: \lambda_i < \underline{\ve}\}} (\lambda_i-y_{ii})^2 \\
	& \geq \sum_{\{i: \lambda_i > \bar{\ve}\}} (\lambda_i-\bar{\ve})^2 
	+ \sum_{\{i: \lambda_i < \underline{\ve}\}} (\lambda_i-\underline{\ve})^2.
	\end{align*}
	
	Thus the lower bound is uniquely attained for $Y \defeq \mathrm{diag}(d_i)$ with
	\begin{equation*}
	d_i \defeq \begin{cases}
	\lambda_i &\mbox{ if } \lambda_i \in [\underline{\ve},\bar{\ve}], \\
	\bar{\ve} &\mbox{ if } \lambda_i > \bar{\ve}, \\
	\underline{\ve} &\mbox{ if } \lambda_i < \underline{\ve}.
	\end{cases}
	\end{equation*} 
\end{proof}

\subsection{Existence}
After given the needed definitions from matrix calculus the goal of this subsection is now to apply \autoref{thm:ex} to the regularization 
functional defined in \autoref{eq:reg} with the set $K= \SPDl_z$
defined in \autoref{eq:K} and 
associated log-Euclidean metric defined below in \autoref{eq:LEmetric}.
Therefore we need to prove the equivalence of the log-Euclidean and Euclidean metric to guarantee in particular that \autoref{ass:metric} \autoref{itm:EquivMetric} is fulfilled. Then \autoref{ass:gc} holds true as stated in \autoref{rem:AssFulfilled} and therefore \autoref{thm:ex} is applicable.

We start by defining and reviewing some properties of the log-Euclidean metric.
\begin{definition}[Log-Euclidean metric]
Let $A,B \in \SPD$. The \emph{log-Euclidean metric} is defined as
\begin{equation}\label{eq:LEmetric}
\boxed{
\d_{\SPD}(A,B) \defeq \d(A,B) \defeq \|\Log(A)-\Log(B)\|_F, \quad A,B \in \SPD.
}
\end{equation}
\end{definition}

\begin{lemma}
	The \emph{log-Euclidean metric} satisfies the metric axioms on $\SPD$.
\end{lemma} 
\begin{proof}
This follows directly because $\norm{\cdot}_F$ is a norm and $\Log$ restricted to $\SPD$ is a diffeomorphism.
\end{proof}

The reasons for choosing this measure of distance is stated in the following remark.
\begin{remark}
The log-Euclidean metric arises when considering $\SPD$
not just as convex cone in the vector space of matrices but as a Riemannian manifold. Thus it can be endowed with a Riemannian metric defined by an inner product on the tangent space, see for example \cite{DryKolZho03,PenFilAya06, FilArsPenAya07a}. 
Two widely used geodesic distances are the \emph{affine-invariant metric} 
\begin{equation}\label{eq:AImetric}
\d_{\mathrm{AI}}(A,B)= \|\Log(A^{-1/2}BA^{-1/2})\|_F, \quad A,B \in \SPD,
\end{equation}
and the \emph{log-Euclidean metric} as stated above. 
These measures of dissimilariy are more adequate in DTMRI as pointed out in \cite{FilArsPenAya07a} because zero or negative eigenvalues induce an infinite distance.

The affine-invariant distance measure is computationally much more demanding which is a major drawback. This is not the case for the log-Euclidean distance, which leads to Euclidean distance computations in the matrix logarithmic domain.
\end{remark}

The following statement can be found in \cite[Section 2.4 \& 2.5]{MinMur18}.
\begin{lemma}
Let $\d =\d_{\SPD}$ denote the log-Euclidean metric (as defined in \autoref{eq:LEmetric}), $\d_{\mathrm{AI}}$ the affine-invariant metric (as defined in \autoref{eq:AImetric}) and $\d_{\R^{m \times m}}$ the standard Euclidean distance. Then $(\SPD,\d)$ as well as $(\SPD,\d_{\mathrm{AI}})$ form a complete metric space. This is not the case for $(\SPD,\d_{\R^{m \times m}})$. 
\end{lemma}

Different metrics induce different properties on the corresponding regularizer. We compare $\Reg_{[\d_{\SPD}]}^l, \Reg_{[\d_{\R^{m \times m}}]}^l$ and additionally a Sobolev seminorm regularizer in the following. 
 
\begin{remark}[Invariances]
For the the log-Euclidean metric $\d =\d_{\SPD}$ as defined in \autoref{eq:LEmetric} the following holds true.
\begin{itemize}
	\item \emph{Scale invariance}: Let $c> 0$ and $A,B \in \SPD$ and denote by $\mathbb{1}_m$ the identity matrix in $\R^{m \times m}$. Then
	\begin{equation*}
	\d(cA,cB) = \d(c\mathbb{1}_m A,c\mathbb{1}_m B) = \norm{\Log(c\mathbb{1}_m) + \log(A) - \Log(c\mathbb{1}_m) - \Log(B)}_F = \d(A,B).
	\end{equation*} 
	\item \emph{Invariance under inversion}: Let $A,B \in \SPD$. Because $\Log(A^{-1}) = -\Log(A)$ we directly get that
	\begin{equation*}
	\d(A^{-1},B^{-1}) = \d(A,B).
	\end{equation*}
	\item \emph{Unitary invariance} Let $A,B \in \SPD$ and $U$ unitary. Then because of the unitary invariance of the Frobenius norm
	\begin{equation*}
	\d(UAU^T,UBU^T) = \d(A,B).
	\end{equation*}
\end{itemize}
These properties transfer to our regularizer $\Reg_{[\d]}^l$ over $W^{s,p}(\Omega;\SPD)$. 
Clearly, when considering $\Reg_{[\d_{\R^{m \times m}}]}^l$, where $\d_{\R^{m \times m}}(A,B) = \norm{A-B}_F$ is the standard Euclidean distance the first two properties do \emph{not} hold true in contrast to the unitary invariance which is also valid. \\
Although we only work with fractional derivatives of order $s \in (0,1)$ we consider for comparison purposes the regularization functional (see also \autoref{eq:FS} in \autoref{sec:experiments})
\begin{equation*}
w \in W^{1,p}(\Omega,\R^{m \times m}) \mapsto \Theta(w) \defeq \int \limits_{\Omega} \norm{\nabla w(x)}_F^p \dx.
\end{equation*}

None of the invariances above, i.e. scale invariance, invariance under inversion and unitary invariance, is valid for $\Theta$. 

Instead, 
\begin{equation*}
\Theta(w+C) = \Theta(w), \quad \Theta(w) = \Theta(-w),
\end{equation*}
for some constant matrix $\ C \in \R^{m \times m}$, i.e. it is translation and reflection invariant. This, in turn, does not hold (or is not even well-defined) for our regularizer $\Reg_{[\d]}^l$ with the log-Euclidean metric but as well when considering the standard Euclidean distance, i.e. it does hold for $\Reg_{[\d_{\R^{m \times m}}]}^l$. A comparison is shown in \autoref{fig:table}.

\begin{figure}[h!]
\begin{center}
	\begin{tabular}{ |c|c|c|c| } 
		\hline
		 & $\Reg^l_{[\d]}$ &  $\Reg^l_{[\d_{\R^{m \times m}}]}$ & $\Theta$ \\ 
		\hline
		scale invariant & \ding{51} & \ding{55} & \ding{55}  \\ 
		\hline
		inversion invariant & \ding{51} & \ding{55} & \ding{55} \\ 
		\hline
		unitary invariant & \ding{51} & \ding{51} & \ding{55} \\ 
		\hline
		translation invariant & \ding{55} & \ding{51} & \ding{51} \\ 
		\hline
		reflection invariant & \ding{55} & \ding{51} & \ding{51} \\ 
		\hline
	\end{tabular}
\end{center}
\caption{Comparison of invariance properties of our regularizer $\Reg^l_{[\d]}, \Reg^l_{[\d_{\R^{m \times m}}]}$ and the regularization term $\Theta$.}
\label{fig:table}
\end{figure}
\end{remark}

In order to show that \autoref{thm:ex} is applicable for $\F_{[\d]}^{\alpha,\vd}$ defined in \autoref{eq:reg} with  $K = \SPDl_z$ and associated log-Euclidean metric $\d = \d_{\SPD}$ defined in \autoref{eq:LEmetric} 
we have to show that \autoref{ass:gc} and therefore \autoref{ass:metric}, in particular the equivalence stated in \autoref{itm:EquivMetric}, is valid. 
In order to prove that we need the following corollary.

\begin{corollary}\label{cor:eigenvalueBounds}
	Let $A \in \SPDl_z$ (defined in \autoref{eq:SPDb}) with eigenvalues $(\lambda_i)_{i=1}^m$. Then for each $i=1, \dots, m$
	\begin{equation}
	\lambda_i \in [\e^{-z},\e^z],
	\end{equation}
	i.e. $\SPDl_z \subset \SPDs_{[\e^{-z},\e^z]}$ (for the definition of latter set see \autoref{eq:SPDs}).
\end{corollary}

\begin{proof}
	If $A \in \SPDl_z$ it holds that $\norm{\Log(A)}_F \leq z$. Using \autoref{eq:FrobeniusEigenvalues} this is equivalent to
	$\sum_{i=1}^{m} \log^2(\lambda_i) \leq z^2$
	so the claim follows.
\end{proof}

Note that the reverse embedding in the previous lemma does \emph{not} hold true. If $A \in \SPDs_{[\e^{-z},\e^z]}$ such that for each eigenvalue $\lambda_i, i=1, \dots, m$ we have that 
$\lambda_i \in \{\e^{-z},\e^z\}$
then $A \in \SPDl_{z\sqrt{m}} \not \subset \SPDl_z$.

Now we can prove that the Euclidean and the log-Euclidean metric are equivalent on $\SPDl_z$.
In particular, we show that $\Log$ is bi-Lipschitz on $\SPDl_z$ and calculate the constant explicitly. Without explicit computation this would follow from the fact that $\Log$ is a diffeomorphism on symmetric, positive definite matrices and that $\SPDl_z$ is a compact subset.

\begin{lemma}\label{lem: Equivalence}
	Let $A,B \in \SPDl_z$ defined in \autoref{eq:SPDb}.
	Then 
	\begin{equation}\label{eq:EquivNorm}
		\frac{1}{\e^z}\norm{A-B}_F^2 \leq \norm{\Log(A)-\Log(B)}_F^2
		\leq \frac{1}{\e^{-z}}\norm{A-B}_F^2.
	\end{equation}
\end{lemma}

\begin{proof}
		
	Since $A$ and $B$ are symmetric and positive definite they can be factorized using their eigendecomposition, see \autoref{rem:SPDProp} \autoref{itm:Decomp}. Hence, we can write
	\begin{equation}\label{eq:ED}
		A = U\Lambda_AU^T, \quad B = V\Lambda_BV^T,
	\end{equation}
	where $U,V \in \R^{m \times m}$ are unitary matrices and $\Lambda_A,\Lambda_B$ are diagonal matrices whose entries are the corresponding positive eigenvalues $(\lambda_1, \dots, \lambda_m)$ of $A$ and $(\mu_1, \dots, \mu_m)$ of $B$, respectively. 
	By \autoref{cor:eigenvalueBounds} it holds that 
	$\lambda_i,\mu_i \in [\e^{-z},\e^z]$
	for all $i = 1, \dots m$.
	
	We consider two cases: 
	
	\textbf{Case 1:} We assume that \emph{all} eigenvalues of $A$ and $B$ are equal, i.e. they have the same one-dimensional spectrum $\mathrm{spec}(A) = \mathrm{spec}(B) = \set{\lambda}$, meaning that $\Lambda \defeq \lambda \mathbb{1}_m \defeq \Lambda_A = \Lambda_B$.	
	This in turn gives that
	\begin{align*}
		&\norm{A-B}_F^2 = \norm{U \Lambda U^T-V \Lambda V^T }_F^2 = \norm{V^T U \Lambda - \Lambda V^T U}_F^2 = \norm{V^T U \lambda - \lambda V^T U}_F^2 = 0, \\
		&\norm{\Log(A)-\Log(B)}_F^2 = \norm{V^T U\log(\Lambda) - \log(\Lambda)V^T U}_F^2 
		= \norm{V^T U\log(\lambda) - \log(\lambda)V^T U}_F^2= 0,
	\end{align*} 

	using the unitary invariance of the Frobenius norm as stated in \autoref{eq:OrthogonalInv} and the properties of the matrix logarithm in \autoref{def:MatrixProp} \autoref{itm:Log} in the second equation. Thus \autoref{eq:EquivNorm} is trivially fulfilled.  
	
	\textbf{Case 2:}
	We now assume that there exists at least two different eigenvalues $\lambda_i \neq \mu_j, \ i,j \in \{1,\dots,m\}$ of $A$ and $B$.
	
	We show the lower inequality 
	$\frac{1}{\e^z}\norm{A-B}_F^2 \leq \norm{\Log(A)-\Log(B)}_F^2$ in \autoref{eq:EquivNorm}. The upper inequality can be done analogously. \\
	By \autoref{eq:OrthogonalInv} and the properties of the matrix logarithm in \autoref{def:MatrixProp} \autoref{itm:Log} it follows that
	\begin{align}\label{eq:Equiv1}
		\norm{\Log(A)-\Log(B)}_F^2 
		&= \norm{U \Log(\Lambda_A) U^T - V \Log(\Lambda_B) V^T}_F^2
		= \norm{V^T U \Log(\Lambda_A) - \Log(\Lambda_B) V^T U}_F^2 \nonumber\\
		&= \norm{C (\diag(\log(\lambda_1), \dots, \log(\lambda_m))) - \diag(\log(\mu_1), \dots, \log(\mu_m)) C}_F^2,
	\end{align}
	where $C \defeq V^TU$.
	Using the definition of the Frobenius norm in \autoref{eq:FrobNorm} we obtain further that
	\begin{equation}\label{eq:Equiv2}
		\norm{C (\diag(\log(\lambda_1), \dots, \log(\lambda_m))) - \diag(\log(\mu_1), \dots, \log(\mu_m)) C}_F^2 =
		\sum_{i,j=1}^{m} \bigg| c_{ij} \big(\log(\lambda_j) - \log(\mu_i) \big) \bigg|^2. 
	\end{equation}
	Indices $(i,j) \in \{1, \dots,m\}$ for which $\lambda_j = \mu_i$ do not contribute to the sum in \autoref{eq:Equiv2} (and do not change the following calculation) so we define $\mathcal{I} \defeq \{ (i,j) \in \{1, \dots,m\} \ : \ \lambda_j \neq \mu_i \}$ as the set of such indices  $(i,j) \in \{1, \dots,m\}$ for which we have $\lambda_j \neq \mu_i$.
	
	From the mean value theorem it follows that for every $(i,j) \in \mathcal{I}$ 
	there exists some 
	\begin{equation*}
	\xi_{ij} \in
		\begin{cases}
			(\lambda_j,\mu_i) \in [\e^{-z},e^z] & \mbox{ if } \lambda_j < \mu_i, \\
			(\mu_i, \lambda_j) \in [\e^{-z},e^z] & \mbox{ if }
			\mu_i < \lambda_j,
		\end{cases}
	\end{equation*}
    such that 
	\begin{equation}\label{eq:Equiv3}
	\sum_{(i,j) \in \mathcal{I}} \bigg| c_{ij} \big(\log(\lambda_j) - \log(\mu_i) \big) \bigg|^2
	= \sum_{(i,j) \in \mathcal{I}} \bigg| c_{ij} \frac{1}{\xi_{ij}} (\lambda_j - \mu_i) \bigg|^2
	\geq \frac{1}{\e^z} \sum_{(i,j) \in \mathcal{I}} \bigg| c_{ij} (\lambda_j - \mu_i) \bigg|^2.
	\end{equation}
	Further we can write
	\begin{align}\label{eq:Equiv4}
	\frac{1}{\e^z} \sum_{(i,j) \in \mathcal{I}} \bigg| c_{ij} (\lambda_j - \mu_i) \bigg|^2
	&= \norm{C (\diag(\lambda_1, \dots, \lambda_m)) - \diag(\mu_1, \dots,\mu_m) C}_F^2 \nonumber \\
	&= \frac{1}{\e^z}  \| C \Lambda_A -\Lambda_B C \|_F^2.
	\end{align}
	Combining \autoref{eq:Equiv1}, \autoref{eq:Equiv2}, \autoref{eq:Equiv3}, \autoref{eq:Equiv4}, the definition of $C = V^TU$ and \autoref{eq:OrthogonalInv} we obtain that
	\begin{equation*}
	\norm{\Log(A)-\Log(B)}_F^2 \geq \frac{1}{\e^z}  \norm{C \Lambda_A -\Lambda_B C}_F^2 = \frac{1}{\e^z}  \norm{U\Lambda_AU^T - V\Lambda_BV^T}_F^2
	=  \frac{1}{\e^z}  \norm{A-B}_F^2
	\end{equation*}
	which finishes the proof.
\end{proof}	

The previous \autoref{lem: Equivalence} proves that \autoref{ass:metric} \autoref{itm:EquivMetric} is valid. This together with \autoref{rem:AssFulfilled} proves the following theorem:

\begin{theorem}\label{thm:ExStabConv}
	Let $K = \SPDl_z$ and $\d = \d_{\SPD}$ as in \autoref{eq:LEmetric}. 
Then the functional $\F_{[\d]}^{\alpha,\vd}$ as defined in \autoref{eq:reg} over $W^{s,p}(\Omega;\SPDl_z)$ satisfies the assertions of \autoref{thm:ex}. In particular, it attains a minimizer and fulfills a stability and convergence result.
\end{theorem}

\subsection{Uniqueness}
\label{ssec:Uniqueness}
So far we showed that the functional $\F_{[\d]}^{\alpha,\vd}$ as defined in \autoref{eq:reg} over $W^{s,p}(\Omega;\SPDl_z)$ using the log-Euclidean metric $\d = \d_{\SPD}$ as in \autoref{eq:LEmetric} attains a minimizer. In this subsection we prove that the minimum is unique. \\
To this end we consider the symmetric, positive definite matrices from a differential geometric point of view. 

The following lemma can be found in \cite{FilArsPenAya07a} and also \cite{PenSomFle19}.
\begin{lemma}
	The space $(\SPD,\d)$ where $\d = \d_{\SPD}$ denotes the log-Euclidean metric as defined in \autoref{eq:LEmetric}
	is a complete, connected Riemmanian manifold with zero sectional curvature.
\end{lemma}

In other words $(SPD,\d)$ is a \emph{flat Hadamard manifold} 
and therefore in particular a \emph{Hadamard space}.
The last property guarantees that the metric $\d$ is \emph{geodesically convex} \cite[Cor. 2.5]{Stu03}, i.e. let $\gamma,\eta: [0,1] \rarr \SPD$ be two geodesics, then
\begin{equation}\label{eq:convex}
\d(\gamma_t,\eta_t) \leq t \d(\gamma_0,\eta_0) + (1-t)\d(\gamma_1,\eta_1).
\end{equation}
Moreover, $\d^p$ is strictly convex in one argument for $p>1$ (\cite[Prop. 2.3]{Stu03} \& \cite [Ex. 2.2.4]{Bac14}), i.e. for $M \in \SPD$ fix and $\gamma_0 \neq \gamma_1$
\begin{equation}\label{eq:strictconvex}
\d^p(\gamma_t,M) < t\d^p(\gamma_0,M) + (1-t)\d^p(\gamma_1,M).
\end{equation}
The following result states that connecting geodesics between two points in $\SPDl_z$ stay in this set.
\begin{lemma}\label{lem:geodesics} Let \autoref{ass:gc} hold.  Let $K = \SPDl_z$ and $\d = \d_{\SPD}$ be the log-Euclidean metric as defined in \autoref{eq:LEmetric}. 
	Let $w^*,w^\diamond \in W^{s,p}(\Omega;\SPDl_z) \subset W^{s,p}(\Omega;\SPD)$. For $\gamma: \Omega \times [0,1] \rarr \SPD$ define
	\begin{equation*}
	\gamma^x \defeq \gamma(x, \cdot) : [0,1] \rarr W^{s,p}(\Omega;\SPD),
	\end{equation*}
	as a connecting geodesic between $\gamma^x(0) = w^*(x)$ and $\gamma^x(1) = w^\diamond(x)$ and
	\begin{equation*}
	\gamma_t \defeq \gamma(\cdot, t) : \Omega \rarr W^{s,p}(\Omega;\SPD), 
	\end{equation*}
	as the evaluation of the geodesic between $w^*(x)$ and $w^\diamond(x)$ at time $t$ for $x \in \Omega$.  
	Then $\gamma_t \in W^{s,p}(\Omega;\SPDl_z)$.
\end{lemma} 

\begin{proof}
	We split the proof into two parts. First we show that $\gamma_t$ maps indeed into $\SPDl_z$. Afterwards we prove that it actually lies in $W^{s,p}(\Omega;\SPDl_z)$. \\
	$\gamma_t$ is a geodesic connecting $\gamma_0(x) = w^*(x)$ and $\gamma_1(x) = w^\diamond(x)$ for $x \in \Omega$. Therefore (\cite[Chapter 3.5]{TurSri16} and \cite{ FilArsPenAya07a}) it can be written as
	\begin{equation*}
	\gamma_t(x) = \Exp\big( t \Log(w^*(x)) + (1-t) \Log(w^\diamond(x)) \big)  
	\end{equation*}
	which is equivalent to
	\begin{equation*}
	\Log(\gamma_t(x)) = t \Log(w^*(x)) + (1-t) \Log(w^\diamond(x)). 
	\end{equation*}
	We denote by $\mathbb{1}_m$ the identity matrix of size $m \times m$ and note that $\norm{\Log(\gamma_t(x))}_F  = \d(\gamma_0(x),\mathbb{1}_m)$, where the Log-Euclidean metric $\d$ is as defined in \autoref{eq:LEmetric}. Because of the geodesic convexity, see \autoref{eq:convex}, we obtain that 
	\begin{equation*}
	\norm{\Log(\gamma_t(x))}_F = \d(\gamma_t(x),\mathbb{1}_m) 
	\leq t \d(\gamma_0(x),\mathbb{1}_m) + (1-t) \d(\gamma_1(x),\mathbb{1}_m) 
	\leq t z + (1-t) z = z
	\end{equation*}
	because $w^*,w^\diamond \in W^{s,p}(\Omega;\SPDl_z)$, i.e. $\norm{\Log(w^*)}_F \leq z, \norm{\Log(w^\diamond)}_F \leq z$.
	 This shows that $\gamma_t$ maps into $\SPDl_z$. 
	
	Next need to prove that actually $\gamma_t \in W^{s,p}(\Omega;\SPDl_z)$, i.e. that
	\begin{align*}
	\norm{\gamma_t}^p_{W^{s,p}(\Omega;\R^{m \times m})} &= 
	\int_{\Omega} \norm{\gamma_t(x)}_F^p \dx +
	\int_{\Omega\times \Omega} \frac{\norm{\gamma_t(x)-\gamma_t(y)}_F^p}{\enorm{x-y}^{n+ps}} \dxy \\
	&= \int_{\Omega} \norm{\gamma_t(x)}_F^p \dx  + \Reg_{[\d_{\R^{m \times m}}]}^0(\gamma_t)
	< \infty.
	\end{align*}
	We
	obtain by Jensen's inequality that
	\begin{equation*}
	\norm{\gamma_t}^p_{W^{s,p}(\Omega;\R^{m \times m})}
	\leq 2^{p-1} \bigg(\int_{\Omega} \norm{\gamma_t(x) - \mathbb{1}_m}_F^p \dx +
	\int_{\Omega} \norm{\mathbb{1}_m}_F^p \dx \bigg)
	+ \Reg_{[\d_{\R^{m \times m}}]}^0(\gamma_t).
	\end{equation*}
	Using \autoref{eq:EquivNorm} it follows that
	\begin{align*}
	&2^{p-1} \bigg(\int_{\Omega} \norm{\gamma_t(x) - \mathbb{1}_m}_F^p \dx +
	\int_{\Omega} \norm{\mathbb{1}_m}_F^p \dx \bigg)
	+ \Reg_{[\d_{\R^{m \times m}}]}^0(\gamma_t) \\
	&\leq 2^{p-1} (\e^z)^{p/2} \bigg( \int_{\Omega} \d^p(\gamma_t(x),\mathbb{1}_m) \dx + \Reg_{[\d]}^0(\gamma_t) 
	\bigg) + C,
	\end{align*}
	where $C \defeq 2^{p-1}|\Omega|$. By using the geodesic convexity stated in \autoref{eq:convex} and \autoref{eq:strictconvex} and again the equivalence of the Euclidean and the log-Euclidean metric (see \autoref{lem: Equivalence}) we get that 
	\begin{align*}
	& 2^{p-1} (\e^z)^{p/2} \bigg( \int_{\Omega} \d^p(\gamma_t(x),\mathbb{1}_m) \dx + \Reg_{[\d]}^0(\gamma_t)
	\bigg) + C \\
	&\leq 2^{p-1} (\e^z)^{p/2} \bigg( t \int_{\Omega} \d^p(\gamma_0(x),\mathbb{1}_m) \dx +
	(1-t) \int_{\Omega} \d^p(\gamma_1(x),\mathbb{1}_m) \dx  \\
	& +
	t \Reg_{[\d]}^0(\gamma_0)
	+ (1-t) \Reg_{[\d]}^0(\gamma_1) 
	\bigg) + C \\
	& \leq 2^{p-1} \e^{pz} 
	\bigg( t \int_{\Omega} \norm{\gamma_0(x)-\mathbb{1}_m}_F^p \dx +
	(1-t) \int_{\Omega} \norm{\gamma_1(x)-\mathbb{1}_m}_F^p \dx  \\
	& +
	t \Reg_{[\d_{\R^{m \times m}}]}^0(\gamma_0)
	+ (1-t) \Reg_{[\d_{\R^{m \times m}}]}^0(\gamma_1)
	\bigg) + C. 
	\end{align*}
	The last expression is finite because of the assumption that $w^*,w^\diamond \in W^{s,p}(\Omega;\SPDl_z)$.
\end{proof}

Now we can state the uniqueness result.
\begin{theorem}\label{thm:uniqueness} Let \autoref{ass:gc} hold.
	 Let $K = \SPDl_z$ and $\d = \d_{\SPD}$ the log-Euclidean metric as defined in \autoref{eq:LEmetric}. Then the functional $\F_{[\d]}^{\alpha,\vd}$ as defined in \autoref{eq:reg} on $W^{s,p}(\Omega;\SPDl_z)$ attains a unique minimizer.
\end{theorem}

\begin{proof}
	Existence of a minimizer is guaranteed by \autoref{thm:ExStabConv}. 
	
	Now, let us assume that there exist two minimizers $w^* \neq w^\diamond \in W^{s,p}(\Omega;\SPDl_z)$ of the functional $\F_{[\d]}^{\alpha,\vd}$. 
	
	Analogously as in \autoref{lem:geodesics} for a geodesic path 
	$\gamma: \Omega \times [0,1] \rarr \SPD$ connecting $w^*$ and $w^\diamond$ 
	we denote by $\gamma_t=\gamma(\cdot,t)$ for $t \in [0,1]$. Thus, in particular, 
	$w^*(x) = \gamma_0(x)$ and $w^\diamond(x) = \gamma_1(x)$ for $x \in \Omega$. Especially,
	$\gamma_t \in W^{s,p}(\Omega;\SPDl_z)$ (see \autoref{lem:geodesics}). \\
	Because $\vd$ is fixed, $\d$ is strictly convex in one argument by \autoref{eq:strictconvex} and convex in both arguments by \autoref{eq:convex} it follows that
	\begin{align}\label{eq:Fconvex}
	\F_{[\d]}^{\alpha,\vd}(\gamma_t) 
	&= \int_\Omega  \chi_{\Omega \setminus D}(x) \d^p(\gamma_t(x),\vd(x)) \dx + 
	\alpha  
	\int_{\Omega\times \Omega} \frac{\d^p(\gamma_t(x), \gamma_t(y))}{\enorm{x-y}^{n+p s}} \rho^l(x-y) \dxy \nonumber \\
	&< t \F_{[\d]}^{\alpha,\vd}(\gamma_0) + (1-t)\F_{[\d]}^{\alpha,\vd}(\gamma_1).
	\end{align}
	Because $w^*$ and $w^\diamond$ are both minimizers we have that 
	$$\F_{[\d]}^{\alpha,\vd}(w^*) = \F_{[\d]}^{\alpha,\vd}(\gamma_0) = 
	\F_{[\d]}^{\alpha,\vd}(\gamma_1) = \F_{[\d]}^{\alpha,\vd}(w^\diamond).$$
	In particular, for $t=1/2$ we obtain by the above equality and by \autoref{eq:Fconvex} that
	\begin{equation*}
	\F_{[\d]}^{\alpha,\vd}(\gamma_{1/2}) < \frac{1}{2} \F_{[\d]}^{\alpha,\vd}(\gamma_0) + \frac{1}{2}\F_{[\d]}^{\alpha,\vd}(\gamma_1)
	= \F_{[\d]}^{\alpha,\vd}(\gamma_0) = \!\!\! \min_{w \in W^{s,p}(\Omega;\SPDl_z)} \F_{[\d]}^{\alpha,\vd}(w),
	\end{equation*}
	which is a contradiction to the minimizing propery of $w^*(x) = \gamma_0(x)$ and $w^\diamond(x) = \gamma_1(x)$ for $x \in \Omega$. Hence, $\gamma_0 $ and $\gamma_1$ must be equal forcing equality in \autoref{eq:Fconvex} and thus giving that the minimum is unique.
\end{proof}

\subsection*{Existence and uniqueness in the case $sp>n$}
If $sp > n$ then existence \emph{and} uniqueness of the minimizer of the functional 
$\F_{[\d]}^{\alpha,\vd}$ even holds on the larger set $W^{s,p}(\Omega;\SPD)$ rather than on  $W^{s,p}(\Omega;\SPDl_z)$, where $\SPD$ is associated with the log-Euclidean
distance $\d = \d_{\SPD}$ as defined in \autoref{eq:LEmetric}. 
Existence in \autoref{thm:ExStabConv} and uniqueness in  \autoref{thm:uniqueness} (with $K = \SPDl_z$) are based on the theory provided in \cite{CiaMelSch19} (see \autoref{thm:ex}) where it is a necessary assumption that the set $K$ is \emph{closed} which is not the case for the set $\SPD$.

Nevertheless it is possible to get existence and uniqueness on this set because 
\begin{center}
for every minimizing sequence $w_k \in W^{s,p}(\Omega;\SPD)$, $k \in \N$, \\
we automatically get that $w_k \in W^{s,p}(\Omega;\SPDl_z)$, 
\end{center}
so that it takes values on the closed subset $\SPDl_z$.
Then, existence of a unique minimizer on $W^{s,p}(\Omega;\SPD)$ follows by the proofs already given, see \cite[Thm. 3.6]{CiaMelSch19} and \autoref{thm:uniqueness}. \\
We now sketch the proof of the assertion. Throughout this sketch $C$ denotes a finite generic constant which, however, can be different from line to line.

	\textbf{Sketch of assertion:} Denote by $\d = \d_{\SPD}$ the log-Euclidean metric (as defined in \autoref{eq:LEmetric}). Let us take a minimizing sequence $w_k \in W^{s,p}(\Omega;\SPD), k \in \N$, of $\F_{[\d]}^{\alpha, \vd}$ so that we can assume that $\F_{[\d]}^{\alpha,\vd}(w_k) \leq C$ for all $w_k, \ k \geq k_0 \in \N$. 

	Computing the log-Euclidean metric leads to evaluations of the Euclidean metric in the matrix logarithmic domain, cf. \autoref{eq:LEmetric}, meaning that
		\begin{equation*} 
		\d(A,B) = \norm{\Log(A) - \Log(B)}_F = \d_{\R^{m \times m}}(\Log(A),\Log(B)),
		\quad A,B \in \SPD.
		\end{equation*}
		This and the fact that $\vd \in L^p(\Omega;\SPD)$ we get that  
		\begin{equation*}
		C \geq \norm{\Log (w_k)}_{L^p(\Omega;\SYM)}^p + 
		       \alpha \Phi_{[d_{\R^{m \times m}}]}^1(\Log (w_k)).
		\end{equation*} 
		Because of \cite[Lemma 2.7]{MelSch20} we can thus bound the $W^{s,p}$-norm of $\Log (w_k)$
		\begin{equation}\label{eq:fractionaSobolevBound}
		C \geq \norm{\Log ( w_k)}_{W^{s,p}(\Omega;\SYM)}^p.
		\end{equation}
		If $sp >n$ the space $W^{s,p}(\Omega;\R^{m \times m})$ is embedded into Hölder-spaces
		$C^{0,\alpha'}(\Omega;\R^{m \times m})$ with $\alpha' \defeq (sp-n)/p$ guaranteed by \cite[Theorem 8.2]{NezPalVal12}. Because of \autoref{eq:fractionaSobolevBound} this gives us that
		\begin{equation*}
		C \geq \norm{\Log (w_k)}_{C^{0,\alpha'}(\Omega;\SYM)}^p,
		\end{equation*}
		yielding in particular that $\norm{\Log (w_k)}_{\infty} < C \defeq z$. 
		
		By the definition of $\SPDl_z$ in \autoref{eq:SPDb} we thus obtain that
		$w_k \in W^{s,p}(\Omega;\SPDl_z)$ for all $k \geq k_0$. 
		Hence, every minimizing sequence $w_k \in W^{s,p}(\Omega;\SPD), k \in \N$, of $\F_{[\d]}^{\alpha, \vd}$ is automatically a minimizing sequence in $W^{s,p}(\Omega;\SPDl_z)$ .
		
\section{Numerics} \label{sec:Numerics}
In this section we go into more detail on the minimization of the regularization functional 
$\F_{[\d]}^{\alpha,\vd}$ defined in \autoref{eq:reg} with the log-Euclidean metric 
$\d = \d_\SPD$ as defined in \autoref{eq:LEmetric} (see \cite{FilArsPenAya07a}) over 
the set $K = \SPDl_z$, the set of symmetric, positive definite $m\times m$ matrices 
with bounded logarithm, as defined in \autoref{eq:K} for denoising and 
inpainting of DTMRI images.

To optimize $\F_{[\d]}^{\alpha,\vd}$ 
we use a projected gradient descent algorithm. The implementation is done in $\mathtt{Matlab}$.
The gradient step is performed by using $\mathtt{Matlabs}$ built-in function $\mathtt{fminunc}$, where the gradient is approximated with a finite difference scheme (central differences in the interior and one-sided differences at the boundary). Therefore, after each step we project the data which are elements of the larger space $\SYM(3)$ back onto $K = \SPDl_z(3)$ by applying the following projection. We remark that by projection we here refer to an idempotent mapping.

\subsection{Projections}
\label{ssec:projections}

\begin{definition}[Projection operators]
	    \quad \\
		$\bullet$ \emph{Projection of $\SYM$ onto $\SPDs_{[\ve,\infty)}$}: 
		Let $M \in \SYM$ be a symmetric matrix with eigendecomposition $M = V \Lambda V^T$ with $\Lambda = \diag(\lambda_1, \dots, \lambda_m)$. Then the projection of $M$ onto the set $\SPDs_{[\ve,\infty)}$ is given by
		\begin{equation}\label{eq:P1}
		\Proj_1:\SYM \rarr \SPDs_{[\ve,\infty)}, \quad
		M \mapsto V\Sigma V^T,
		\end{equation}
		where $\Sigma = \diag(\mu_1, \dots, \mu_m)$ with
		\begin{equation*}
		\mu_i \defeq \begin{cases}
		\lambda_i &\mbox{ if } \lambda_i \geq \ve,\\
		\ve &\mbox{ if } \lambda_i < \ve.
		\end{cases}
		\end{equation*}
		
		$\bullet$ \emph{Projection of $\SPDs_{[\ve,\infty)}$ onto $\SPDl_z$}: 
		Let $M \in \SPDs_{[\ve,\infty)}$ with eigenvalues $(\lambda_i)_{i=1}^m$ and eigendecomposition $M = V\Lambda V^T$. Define 
		$C_{\mathrm{Frob}} \defeq \norm{\Log(M)}_F^2 = \sum_{i=1}^m \log^2(\lambda_i)$ as the squared Frobenius norm of $\Log(M)$. Then the projection of $M$ onto $\SPDl_z$ is given by
		\begin{equation}\label{eq:P2}
		\Proj_2: \SPDs_{[\ve,\infty)} \rarr \SPDl_z, \quad
		M \mapsto V\Sigma V^T,
		\end{equation}
		where $\Sigma = \diag(\mu_1, \dots, \mu_m)$ with
		\begin{equation*}
		\mu \defeq \begin{cases}
		\lambda &\mbox{ if } C_{\mathrm{Frob}} \leq z^2,\\
		\lambda^{z/\sqrt{C_{\mathrm{Frob}}}} &\mbox{ if } C_{\mathrm{Frob}} > z^2,
		\end{cases}
		\end{equation*}
        where $\lambda = (\lambda_1, \dots, \lambda_m)^T$ and $\mu = (\mu_1, \dots, \mu_m)^T$ are the vectors containing all eigenvalues.
        
        $\bullet$ \emph{Projection of $\SYM$ onto $\SPDl_z$}: Let $M \in \SYM$ be a symmetric matrix. We define its projection $\Proj(M)$ onto $\SPDl_z$ as 
        \begin{equation}\label{eq:P}
        \Proj: \SYM \rarr \SPDl_z, \quad
        M \mapsto \Proj_2 (\Proj_1(M)).
        \end{equation}
\end{definition}

For a given matrix $M \in \SYM$ the projection $\Proj_1(M) \in \SPDs_{[\ve,\infty)}$ is the closest approximation in the Frobenius norm, 
i.e. 
\begin{equation}
	\Proj_1(M) = \argmin_{X \in \SPDs_{[\ve,\infty)}} \|M-X\|_F^2,
\end{equation}
as stated in \autoref{lem:Approx} when choosing $\underline{\ve} = \ve$ and $\bar{\ve} = \infty$.

If $M \in \SPDs_{[\ve,\infty)}$ the projection $\Proj_2$  scales the eigenvalues of $M$ in such a way that it is guaranteed that $\norm{\Log(\Proj_2(M))}_F \leq z$, i.e.
$\Proj_2(M) \in \SPDl_z$.

In fact, if $C_{\mathrm{Frob}} = \norm{\Log(M)}_F^2  > z^2$ meaning that 
$M \not \in \SPDl_z$
then
\begin{equation*}
\norm{\Log(\Proj_2(M))}_F^2 = \sum_{i=1}^m \log^2(\lambda_i^{z/\sqrt{C_{\mathrm{Frob}}}})  = \frac{z^2}{C_{\mathrm{Frob}}}\sum_{i=1}^m \log^2(\lambda_i) 
= \frac{z^2}{C_{\mathrm{Frob}}}\underbrace{\norm{\Log(M)}_F^2}_{C_{\mathrm{Frob}}}
= z^2,
\end{equation*}
giving that $\Proj_2(M) \in \SPDl_z$. 

The following lemma shows that the projected functions stay in the same regularity class.
\begin{lemma}\label{lem:inclusion}
	\begin{enumerate}
		\item \label{itm:inclP1} Let $w \in W^{s,p}(\Omega;\SYM)$. Then $\Proj_1 (w) \in W^{s,p}(\Omega;\SPDs_{[\ve,\infty)})$.
		\item \label{itm:inclP2} Let $w \in W^{s,p}(\Omega;\SPDs_{[\ve,\infty)})$. Then $\Proj_2 (w) \in W^{s,p}(\Omega;\SPDl_z)$.
		\item \label{itm:inclP} Let $w \in W^{s,p}(\Omega;\SYM)$. Then $\Proj  (w) \in W^{s,p}(\Omega;\SPDl_z)$.
	\end{enumerate}	
\end{lemma}

\begin{proof}
	\begin{enumerate}
		\item 
		Let $w \in W^{s,p}(\Omega;\SYM(m))$ and define $v \defeq \Proj_1 (w)$. By the definition of $\Proj_1$
		it follows directly that $v: \Omega \rarr \SPDs_{[\ve,\infty)}(m)$. 
		
		By \autoref{rem:SPDProp} \autoref{itm:Decomp} and the definition of $\Proj_1$ (see \autoref{eq:P1} and also \autoref{lem:Approx}) we can decompose $w(x)$ and $v(x)$ for $ x \in \Omega$ as follows:
		\begin{equation*}
		w(x) = R(x)W(x)R^T(x), \ v(x) = R(x)V(x)R^T(x),
		\end{equation*}
		with orthonormal matrix $R \in \R^{m \times m}$ and diagonal matrices $W,V \in \R^{m \times m}$. Denote the eigenvalues of $w(x)$ as $(\lambda_i(x))_{i=1}^m$. 
		The eigenvalues of $v(x)$ are then defined as
		\begin{equation}\label{eq:eigenvaluesProj}
		\mu_i(x) \defeq \begin{cases}
		\lambda_i(x) &\mbox{ if } \lambda_i(x) \in [\ve,\infty), \\
		\ve &\mbox{ if } \lambda_i(x) < \ve,
		\end{cases} \quad i=1, \dots m.
		\end{equation}
		
		It remains to show that $v \in W^{s,p}(\Omega;\SPDs_{[\ve,\infty)}(m))$, i.e.
		\begin{equation}\label{eq:vWSP}
		\norm{v}^p_{W^{s,p}(\Omega;\R^{m \times m})} = 
		\int_{\Omega} \norm{v(x)}_F^p \dx +
		\int_{\Omega\times \Omega} \frac{\norm{v(x)-v(y)}_F^p}{\enorm{x-y}^{n+ps}} \dxy < \infty.
		\end{equation}

		We start to bound the $L^p$-norm in \autoref{eq:vWSP}. \\
		For each $x \in \Omega$ it holds that 
		\begin{equation}\label{eq:Lp1}
		\norm{v(x)}_F^p = \norm{R(x)V(x)R^T(x)}_F^p = \norm{V(x)}_F^p
		= \left( \sum_{i=1}^{m} \enorm{\mu_i(x)}^2 \right)^{p/2},
		\end{equation}
		using the unitary invariance of the Frobenius norm, see \autoref{eq:OrthogonalInv}. \\ 
		From the definition of the $m$ eigenvalues in \autoref{eq:eigenvaluesProj} and Jensen's inequality it follws that
		\begin{equation}\label{eq:splitting}
		\begin{aligned}
		\left( \sum_{i=1}^{m} \enorm{\mu_i(x)}^2 \right)^{p/2}
		&= \left( \sum_{\{i:\mu_i = \ve\}} \enorm{\mu_i(x)}^2  
		+  \sum_{\{\i:\mu_i> \ve\}} \enorm{\mu_i(x)}^2 \right)^{p/2} \\
		&\leq 2^{p-1} \left(m\ve^2\right)^{p/2} + 2^{p-1} \left( \sum_{i=1}^{m} \enorm{\lambda_i(x)}^2 \right)^{p/2}. 
		\end{aligned}
		\end{equation}
		We thus obtain using \autoref{eq:Lp1} and \autoref{eq:splitting} that
		\begin{align}\label{eq:Lpbound}
		\norm{v}_{L^p(\Omega;\R^{m \times m})}^p 
		& = \int \limits_{\Omega} \left( \sum_{i=1}^{m} \enorm{\mu_i(x)}^2 \right)^{p/2} \dx
		\leq 2^{p-1} \int \limits_{\Omega} \left(m\ve^2\right)^{p/2} + \left( \sum_{i=1}^{m} \enorm{\lambda_i(x)}^2 \right)^{p/2} \dx\\
		&\leq 2^{p-1}\left(m\ve^2\right)^{p/2}|\Omega| + 2^{p-1}\norm{w}_{L^p(\Omega;\R^{m \times m})}^p < \infty,
		\end{align}
		because $\Omega$ is bounded and $w \in W^{s,p}(\Omega;\SYM(m))$, in particular $w \in L^p(\Omega;\SYM(m))$.
		
		The $W^{s,p}$-semi-norm in \autoref{eq:vWSP} can be bounded in a similar way.\\
		Therefore we calculate for $x,y \in \Omega$
		\begin{align}\label{eq:FrobNormEstimate1}
		\norm{v(x)-v(y)}_F^p &= \norm{R(x)V(x)R^T(x)-R(y)V(y)R^T(y)}_F^p
		=  \norm{R^T(y)R(x)V(x)-V(y)R^T(y)R(x)}_F^p \nonumber \\
		&\eqdef \norm{C(x,y)V(x)-V(y)C(x,y)}_F^p = 
		\left( \sum_{i,j=1}^{m} c_{ij}(x,y)\big(\mu_j(x) - \mu_i(y)\big)^2 \right)^{p/2},
		\end{align}
		where the last equality holds true because the matrix $C$ is symmetric.
		The same calculation is valid for $w$ so that we obtain
		\begin{equation}\label{eq:FrobNormEstimate2}
		\norm{w(x)-w(y)}_F^p =
		\left( \sum_{i,j=1}^{m} c_{ij}(x,y)\big(\lambda_j(x) - \lambda_i(y)\big)^2 \right)^{p/2}.
		\end{equation}
		By the definition of the eigenvalues $\mu$, see \autoref{eq:eigenvaluesProj}, and by using a splitting of the sum as in \autoref{eq:splitting} it can be shown that
		\begin{equation}\label{eq:BoundforWsp}
		\left( \sum_{i,j=1}^{m} c_{ij}(x,y) \big(\mu_j(x) - \mu_i(y)\big)^2 \right)^{p/2} \leq
		\left( \sum_{i,j=1}^{m} c_{ij}(x,y) \big(\lambda_j(x) - \lambda_i(y)\big)^2 \right)^{p/2}.
		\end{equation}
		This implies that (using  \autoref{eq:FrobNormEstimate1}, \autoref{eq:FrobNormEstimate2}, \autoref{eq:BoundforWsp} )
		\begin{align*}
		\enorm{v}_{W^{s,p}(\Omega;\R^{m \times m})} 
		&= \int_{\Omega\times \Omega} \frac{\norm{v(x)-v(y)}_F^p}{\enorm{x-y}^{n+ps}} \dxy 
		= \int_{\Omega\times \Omega} \frac{\left( \sum_{i,j=1}^{m} c_{ij}(x,y) \big(\mu_j(x) - \mu_i(y)\big)^2 \right)^{p/2}}{\enorm{x-y}^{n+ps}} \dxy \\
		& \leq \int_{\Omega\times \Omega} \frac{\left( \sum_{i,j=1}^{m} c_{ij}(x,y) \big(\lambda_j(x) - \lambda_i(y)\big)^2 \right)^{p/2}}{\enorm{x-y}^{n+ps}} \dxy 
		= \enorm{w}_{W^{s,p}(\Omega;\R^{m \times m})} < \infty
		\end{align*}
		because of the fact that $w \in W^{s,p}(\Omega;\SYM(m))$.
		\item The proof can be done similar to the previous one in \autoref{itm:inclP1}.
		\item This is a direct consequence of \autoref{itm:inclP1} and \autoref{itm:inclP2}.
	\end{enumerate}
\end{proof}

The next lemma shows that minimizing elements of $\F_{[\d]}^{\alpha,\vd}$ on $W^{s,p}(\Omega;\SPDl_z)$ and minimizing elements of the projected gradient method are connected.
Therefore we define an extension of $\F_{[\d]}^{\alpha,\vd}$ to the larger space $W^{s,p}(\Omega;\SYM)$ as
$\Ft_{[\d]}^{\alpha,\vd}: W^{s,p}(\Omega;\SYM) \rarr [0,\infty)$ given by
\begin{equation*} 
		\begin{aligned} 
			\Ft_{[\d]}^{\alpha,\vd}(u) \defeq
			& \int_\Omega \chi_{\Omega \setminus D}(x) \d^p(\Proj(u(x)),\vd(x)) \dx + 
			\alpha \tilde{\Reg}_{[\d]}^{l} (u), \\
			\text{ with } \tilde{\Reg}^{l}_{[\d]}(u) \defeq & 
			\int_{\Omega\times \Omega} \frac{\d^p(\Proj(u(x)), \Proj(u(y)))}{\enorm{x-y}^{n+p s}} \rho^l(x-y) \dxy, 
	\end{aligned}
\end{equation*} 
with the projection operator $P: W^{s,p}(\Omega;\SYM) \rarr  W^{s,p}(\Omega;K) =  W^{s,p}(\Omega;\SPDl_z)$ defined in \autoref{eq:P}. 

\begin{lemma}
Let $K = \SPDl_z$ and let $\d = \d_{\SPD}$ be the log-Euclidean metric as defined in \autoref{eq:LEmetric}.
\begin{enumerate}
	\item \label{itm:minFt} Let $w^* \in \mathrm{argmin}_{w \in W^{s,p}(\Omega;\SPDl_z)} \F_{[\d]}^{\alpha,\vd}(w)$. Then in particular $w^* \in W^{s,p}(\Omega;\SYM)$ and it is a minimizer of $\Ft_{[\d]}^{\alpha,\vd}$, i.e. $w^* \in \mathrm{argmin}_{u \in W^{s,p}(\Omega;\SYM)} \Ft_{[\d]}^{\alpha,\vd}(u)$. 
	\item \label{itm:minF} 
		Let $u^* \in \mathrm{argmin}_{u \in W^{s,p}(\Omega;\SYM)} \Ft_{[\d]}^{\alpha,\vd}(u)$. Then $w^* \defeq \Proj(u^*) \in W^{s,p}(\Omega;\SPDl_z)$ is a minimizer of $\F_{[\d]}^{\alpha,\vd}$, i.e. $w^* \in \mathrm{argmin}_{w \in W^{s,p}(\Omega;\SPDl_z)} \F_{[\d]}^{\alpha,\vd}(w)$. 
\end{enumerate}
\end{lemma}
The proof is straightforward. \\
\begin{remark}
Basically, the second item of the previous lemma shows that
\begin{equation*}
\mathrm{argmin}_{u \in W^{s,p}(\Omega;\SYM)} \F_{[\d]}^{\alpha,\vd} \circ \Proj
= \Proj^{-1} \big( \mathrm{argmin}_{w \in W^{s,p}(\Omega;\SPDl_z)} \F_{[\d]}^{\alpha,\vd} \big).
\end{equation*}
\end{remark}
\color{black}

\section{Numerical experiments}
\label{sec:experiments}
After clarifying existence, uniqueness, stability and convergence of variational regularization methods in an infinite dimensional function set setting, we move to the discretized optimization problems, which are finite dimensional optimization problems on manifolds.
In order to present and evaluate our numerical experiments, we need a method of comparison, 
which is outlined in \autoref{ss:cf} and a quality criterion, which is described in 
\autoref{ss:qc}. We present experiments with synthetic and real data in \autoref{ssec:results}. The generation of synthetic data is described in \autoref{ssec:noisyData}. 

When minimizing $\F_{[\d]}^{\alpha,\vd}$ we follow the concept of discretize-then-optimize.
So, in the text below, when we talk about numerical implementation the functional should always be considered as a discretized functional on a finite dimensional subset of $W^{s,p}(\Omega;\SPDl_z)$. Nevertheless, we write the functional as it is defined in the infinite dimensional setting.
However, we recall again the fundamental difference between the infinite dimensional setting and the discretized one: After discretization the functional deals with mappings from 
a vector (with dimension of the numbers of pixel) into a product vector of manifold-valued components, that is an optimization problem on manifolds. Such a formulation is not possible for the infinite dimensional one.

The numerical results build up on the following parameter setting:
\begin{enumerate}
	\item In the concrete examples in \autoref{ssec:results} we take $m=3$ and $n=2$. This means that we manipulate (denoise and inpaint) a 2-dimensional slice of a 
	3-dimensional DTMRI image. 
	\item In the regularization term $\Reg^l_{[\d]}$, defined in \autoref{eq:phi} 
	we choose $l=1$ in order to take advantage of the locally supported mollifier, see \autoref{def:mol}.
\end{enumerate}

\subsection{Optimization} 
\label{ssec:optimization}
As described in the previous \autoref{ssec:projections} when optimizing the functional
$\F_{[\d]}^{\alpha,\vd}$ (defined in \autoref{eq:reg} and $\d =\d_{\SPD}$ defined in \autoref{eq:LEmetric} ) we use a projected gradient descent algorithm by applying the projection 
$\Proj = \Proj_2 \circ \Proj_1$ to each diffusion tensor after each step (as defined in \autoref{eq:P}). \\
$\Proj_1$ first projects onto the set $\SPDs_{[\ve,\infty)}(3)$.
In the implementation we used $\ve = \mathtt{eps}$, where $\mathtt{eps}$ is the floating-point relative accuracy in $\mathtt{Matlab}$. Then $\Proj_2$ projects onto $\SPDl_z(3)$, where we used $z=36$. This is due to the fact that if $A \in \SPDl_{36}(3)$ then its eigenvalues lie in the interval $[\e^{-36},\e^{36}] \approx [\mathtt{eps},\e^{36}]$, see \autoref{cor:eigenvalueBounds}, so that we are able to compute diffusion tensors close to zero without projecting them. A summary of parameters used is shown in \autoref{fig:table2}.

The (discrete) mollifier $\rho$ in \autoref{eq:reg} (we choose $l=1$) is defined such a way that it has non-zero support on \emph{up to nine} neighboring pixels in \emph{each} direction. 
The number of non-zero elements is denoted by $n_\rho$ and we refer to \autoref{fig:rho} for an illustration. 
The function $\rho: \R^n \rarr \R$ satisfies two needs: One the one hand it allows to combine two different concepts. The characterization theory of \cite{BouBreMir01} and a classical theory of Sobolev spaces \cite{Maz85}. On the other hand, there is a practical aspect, which is related to computation time: The smaller the essential support of $\rho$ is, the faster the optimization algorithm can be implemented. In other words, a large support of would be desired for a quasi Sobolev norm regularization implementation but this hinders a very efficient implementation. 

\begin{figure}[h]
	\begin{center}
		\begin{tikzpicture}
			\draw(-3.5,-3.5) grid (3.5,3.5);
			\draw[thick] (0,0) circle (0.2);
			\fill[black] (1,1) + (0, 0.2) arc (90:270:0.2);
			\fill[gray] (1,1) + (0, -0.2) arc (270:450:0.2);
			
			\fill[black] (1,0) + (0, 0.2) arc (90:270:0.2);
			\fill[gray] (1,0) + (0, -0.2) arc (270:450:0.2);
			
			\fill[black] (1,-1) + (0, 0.2) arc (90:270:0.2);
			\fill[gray] (1,-1) + (0, -0.2) arc (270:450:0.2);
			
			\fill[black] (0,-1) + (0, 0.2) arc (90:270:0.2);
			\fill[gray] (0,-1) + (0, -0.2) arc (270:450:0.2);
			
			\fill[black] (-1,-1) + (0, 0.2) arc (90:270:0.2);
			\fill[gray] (-1,-1) + (0, -0.2) arc (270:450:0.2);
			
			\fill[black] (-1,0) + (0, 0.2) arc (90:270:0.2);
			\fill[gray] (-1,0) + (0, -0.2) arc (270:450:0.2);
			
			\fill[black] (-1,1) + (0, 0.2) arc (90:270:0.2);
			\fill[gray] (-1,1) + (0, -0.2) arc (270:450:0.2);
			
			\fill[black] (0,1) + (0, 0.2) arc (90:270:0.2);
			\fill[gray] (0,1) + (0, -0.2) arc (270:450:0.2);
			
			\filldraw[black](0,2)circle[radius=0.2];
			\filldraw[black](1,2)circle[radius=0.2];
			\filldraw[black](2,2)circle[radius=0.2];
			\filldraw[black](2,1)circle[radius=0.2];
			\filldraw[black](2,0)circle[radius=0.2];
			\filldraw[black](2,-1)circle[radius=0.2];
			\filldraw[black](2,-2)circle[radius=0.2];
			\filldraw[black](1,-2)circle[radius=0.2];
			\filldraw[black](0,-2)circle[radius=0.2];
			\filldraw[black](-1,-2)circle[radius=0.2];
			\filldraw[black](2,-1)circle[radius=0.2];
			\filldraw[black](-2,-2)circle[radius=0.2];
			\filldraw[black](-2,-1)circle[radius=0.2];
			\filldraw[black](-2,0)circle[radius=0.2];
			\filldraw[black](-2,1)circle[radius=0.2];
			\filldraw[black](-2,2)circle[radius=0.2];
			\filldraw[black](-1,2)circle[radius=0.2];
		\end{tikzpicture}
		\caption{Support of the discrete mollifier $\rho$ with $n_{\rho} = 1$ (gray) and $n_{\rho} = 2$ (black) when centered at the unfilled point in the middle. In the examples we have chosen $n_{\rho} = 9$ at most.}
		\label{fig:rho}
	\end{center}
\end{figure}
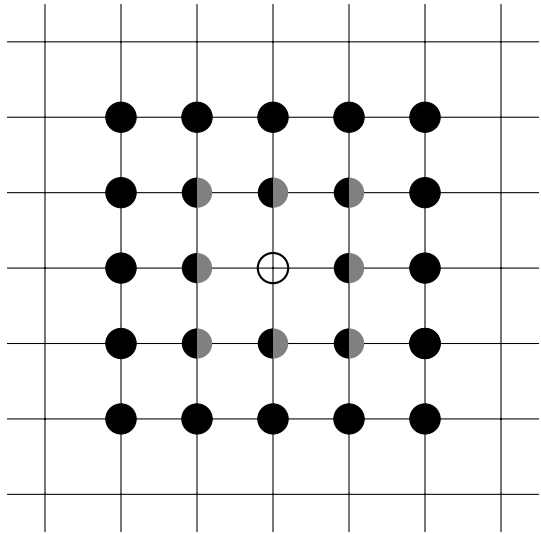

\subsection{Comparison functional} \label{ss:cf}
We compare the results with the ones obtained by optimizing the \emph{comparison functional} $\F_{\text{C}}$ defiend as
\begin{equation}\label{eq:FS}
\boxed{
\F_{\text{C}}(w) \defeq \int_\Omega \chi_{\Omega \setminus D}(x) \norm{w(x)- \vd(x)}_F^p \dx + \beta \int \limits_{\Omega} \norm{\nabla w(x)}_F^p \dx
}
\end{equation}
on $W^{1,p}(\Omega;\SPDl_z) \subset W^{1,p}(\Omega;\R^{m \times m})$ (see \cite[Cor 5.5]{NezPalVal12}). Here, the fidelity term consists of the $L^p$-norm while the regularizer is the vectorial Sobolev semi-norm to the power $p$.
In the implementation we project the data back onto $K = \SPDl_{36}(3)$ after each gradient step as described before.

\subsection{Measure of quality} \label{ss:qc}
As a measure of quality we compute the \emph{signal-to-noise ratio (SNR)} which is defined as 
\begin{equation*}
SNR = \frac{ \norm{w^{\mathrm{orig}}}_F}{\norm{w^{\mathrm{orig}} - w^{\mathrm{rec}}q}_F},
\end{equation*}
where $w^{\mathrm{orig}}$ describes the ground truth and $w^{\mathrm{rec}}$ the reconstructed data. 

\subsection{Noisy data generation} 
\label{ssec:noisyData}
We consider a discretized version of $\Omega \subset \R^2$ as a quadratic grid of size $N \times N, N \in \N$ with equally distributed pixels $(p^{i,j})_{i,j=1}^N$. On each $p^{i,j}$ a symmetric, positive definite diffusion tensor $w^{i,j} \in \R^{3 \times 3}$ (with bounded logarithm) is located describing the underlying diffusion process in the biological tissue. \\
In DTMRI the data that are actually measured are so-called \emph{diffusion weighted images} (DWIs) $(A_{[b,g]}(p^{i,j}))_{i,j=1}^N$. They describe the diffusion in a direction $g \in \R^3$ with given b-value $b \in \R$ at a pixel $p^{i,j}$. The diffusion tensor and the DWIs are related by the \emph{Stejskal-Tanner equation} \cite{Ste65,SteTan65,BasLebMat94a}: 
\begin{equation}\label{eq:ST}
A_{[b,g]}(p^{i,j}) = A_0 \e^{-bg^T w^{i,j} g} 
\end{equation}
for all pixels $p^{i,j}$, where we assume that $A_0 \in \R_{\geq 0}$ is known. For more details and a survey on MRI see for example \cite{Jon11}.

To generate our noisy synthetic data $(\vd)^{i,j}$ we computed 12 DWIs $(A^{1}_{[b,g]}(p^{i,j}), \dots, A^{12}_{[b,g]}(p^{i,j}))$ from our initial (original) synthetic diffusion tensor (a symmetric, positive definite matrix with bounded logarithm) $w^{i,j}$ on each pixel $p^{i,j}$ via \autoref{eq:ST}. Then we imposed Rician noise on them (\cite{GudPat95,BasFleWhi06}) with different values of $\sigma^2$. 
We used a least squares fitting (as described shortly in \cite{TscDer04}) followed by the projection $\Proj$ to obtain a noisy diffusion tensor image on each pixel such that $(\vd)^{i,j} \in \SPDl_z(3)$ for $i,j \in \{1, \dots, N\}$.

In the synthetic examples in \autoref{subsubsec:syndenoising} and \autoref{subsubsec:syninpainting} we chose $A_0 = 1000$ and $b = 800$ to generate the noisy data. 
The real data set in \autoref{subsubsec:realdenoising} is freely accessible (\cite{CabAndBasMai}) and provides corresponding values of $A_0$ and $b$. For an overview of parameters see \autoref{fig:table2}.

\subsection{Visualization} 
\label{ssec:visualization}

On each pixel $(p^{i,j})_{i,j=1}^N$ the diffusion process is described by a a symmetric, positive definite diffusion tensor $w^{i,j} \in \R^{3 \times 3}$ (with bounded logarithm). We visualize it by a 3D ellipsoid. Therefore we take the (normed) eigenvectors $v_1^{i,j},v_2^{i,j},v_3^{i,j}$ and the corresponding eigenvalues $\lambda_1^{i,j},\lambda_2^{i,j},\lambda_3^{i,j}$ and interpret the eigenvectors as axis of an ellipsoid with length $\lambda_1, \lambda_2$ and $\lambda_3$, respectively.\\
We color the ellipsoids corresponding to the value of its \emph{fractional anisotropy FA} defined as
\begin{equation}\label{eq:FA}
	FA^{i,j} \defeq \sqrt{\frac{(\lambda_1^{i,j}-\lambda_2^{i,j})^2 + (\lambda_2^{i,j}-\lambda_3^{i,j})^2 + (\lambda_1^{i,j}-\lambda_3^{i,j})^2 }{2(\lambda_1^{i,j} \lambda_1^{i,j} + \lambda_2^{i,j} \lambda_2^{i,j} + \lambda_3^{i,j} \lambda_3^{i,j})}}, \quad i,j \in \{1, \dots,N\}.
\end{equation}
Fractional anisotropy is an index between $0$ and $1$ for measuring the amount of anisotropy within a pixel.
If there is no anisotropy, i.e. if the ellipsoid is sphere-shaped, then all eigenvalues are equal and the fractional anisotropy is zero, which we color black. The higher the value of $FA$ within a pixel the lighter blue we color the ellipsoid. A colorscale is illustrated in \autoref{fig:colorscale}. 

\begin{figure}[!h]
	\centering
	\includegraphics[width=0.4\linewidth]{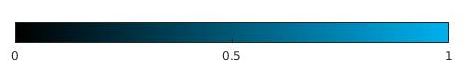}
	\caption{Colorscale used in the numerical results. The value between $0$ and $1$ represent the fractional anisotropy of each ellipsoid. Here, the value zero describes a sphere.}
	\label{fig:colorscale}
\end{figure}

\subsection{Numerical results} 
\label{ssec:results}
Now we present concrete numerical examples for denoising and inpainting of diffusion tensor images. 
The diffusion tensors are represented via ellipsoids as described in \autoref{ssec:visualization}.
The parameters used are summarized in the following table.
\begin{figure}[h!]
	\begin{center}
		\begin{tabular}{ |c|c| } 
			\hline
			parameter&  value\\ 
			\hline
			$\ve$ & $\mathtt{eps}$  \\ 
			\hline
			$z$ & 36 \\ 
			\hline
			$A_0$ & 1000 \\ 
			\hline
			$b$ & 800 \\ 
			\hline
			$l$ & 1 \\ 
			\hline
			$n$ & 2 \\ 
			\hline
			$m$ & 3 \\ 
			\hline
		\end{tabular}
	\end{center}
	\caption{Parameters and corresponding values used in the numerical examples.}
	\label{fig:table2}
\end{figure}
Note that the values of $A_0$ and $b$ are only valid for the synthetic data sets; in the real data set in \autoref{fig:fig4} these values are provided.

\subsubsection{Denoising of synthetic data}
\label{subsubsec:syndenoising}
The \textbf{first example} is represented in \autoref{fig:fig1} and concerns denoising of a synthetic image in $W^{s,p}(\Omega,\SPDl_{36}(3))$. The motivation of the choice $z=36$ was explained in the previous \autoref{ssec:optimization}. 

The noisy image is obtained by adding Rician noise to the corresponding DWIs with $\sigma^2 = 40$ as described in \autoref{ssec:noisyData}.

The original image is shown in \autoref{sfig:sfig1-a}. In a column all ellipsoids have the same shape. In the first column the ellipsoids shown are sphere-shaped, i.e. all eigenvalues are equal with a value of $0.5 \cdot 10^{-3}$.  The fractional anisotropy (see \autoref{eq:FA}) is zero and hence these ellipsoids are colored black, see \autoref{fig:colorscale}. Going from the first column to the last one one eigenvalue is increasing from $0.5 \cdot 10^{-3}$ to $3.5 \cdot 10^{-3}$ while the other two stay constant. This leads to an increasing value of the fractional anisotropy and thus to a light blue coloring, see also \autoref{fig:colorscale}. The averaged value (over the column) of the increasing eigenvalue is plotted in black in \autoref{sfig:sfig1-f}.

The results obtained by using our metric double integral regularization (see \autoref{eq:reg}) can be seen in \autoref{sfig:sfig1-c} while the results using Sobolev-semi-norm regularization (see \autoref{eq:FS}) are illustrated in \autoref{sfig:sfig1-d} and \autoref{sfig:sfig1-e}. 
Our method removes the noise while the size of the ellipsoids stays close to the size of them in the original image.
This is in particular visible in \autoref{sfig:sfig1-f}, where the averaged size of the increasing eigenvalue is plotted in red.
Choosing the parameter $\beta$ in the Sobolev semi-norm regularization term too small results in a quite noisy image while a larger value of $\beta$ smooths the whole image which can be seen particularly on the left-hand-side where the ellipsoids are quite tiny. The smoothing effect is even more visible in \autoref{sfig:sfig1-f}.   

\begin{figure}[!h]
	\centering
\begin{subfigure}[h]{0.45\linewidth}
	\includegraphics[width=1\linewidth]{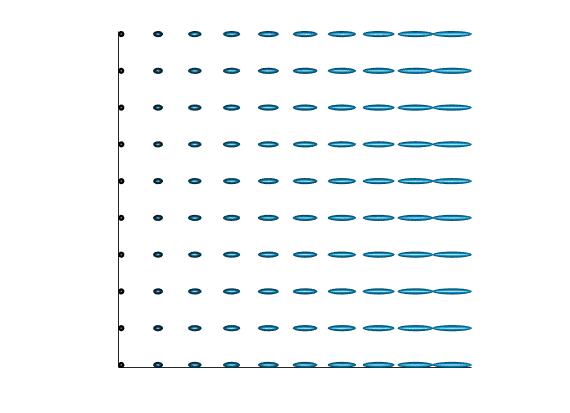}
	\caption{Original data.}
	\label{sfig:sfig1-a}
\end{subfigure}
\hspace{1em}
\begin{subfigure}[h]{0.45\linewidth}
	\includegraphics[width=1\linewidth]{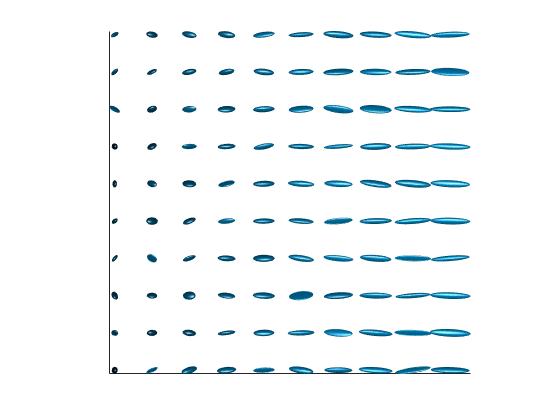}
	\caption{Noisy data using $\sigma^2 = 40$.}
	\label{sfig:sfig1-b}
\end{subfigure}

\begin{subfigure}[h]{0.45\linewidth}
	\includegraphics[width=1\linewidth]{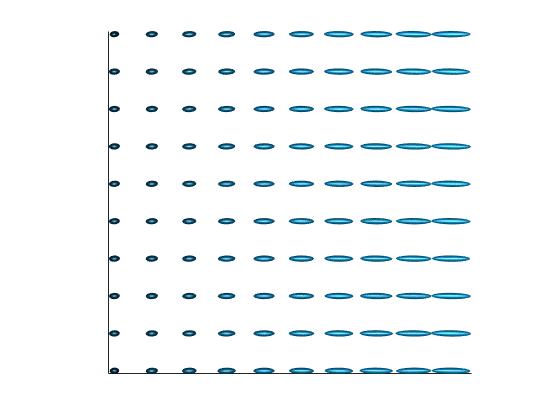}
	\caption{Result with metric double integral regularization with $\alpha=1$, $SNR = 21.02$.}
	\label{sfig:sfig1-c}
\end{subfigure}
\hspace{1em}
\begin{subfigure}[h]{0.45\linewidth}
	\includegraphics[width=1\linewidth]{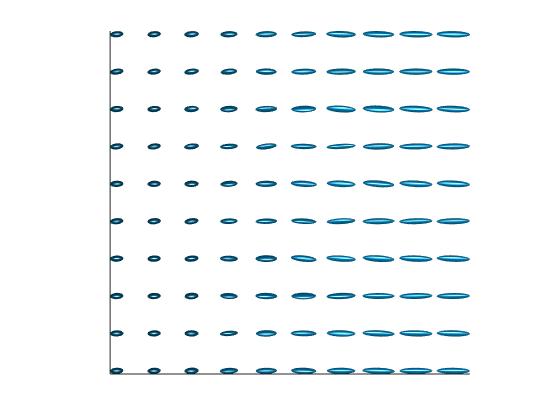}
	\caption{Sobolev semi-norm regularization with $\beta=2$, $SNR = 7.95$.}
	\label{sfig:sfig1-d}
\end{subfigure}

\begin{subfigure}[h]{0.45\linewidth}
	\includegraphics[width=1\linewidth]{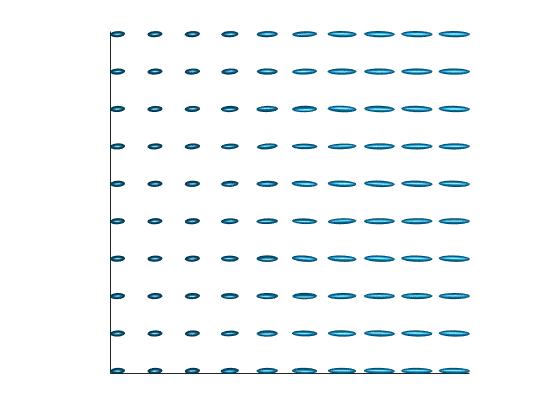}
	\caption{Sobolev semi-norm regularization with $\beta=3$, $SNR = 6.92$.}
	\label{sfig:sfig1-e}
\end{subfigure}
\hspace{1em}
\begin{subfigure}[h]{0.45\linewidth}
	\includegraphics[width=1\linewidth]{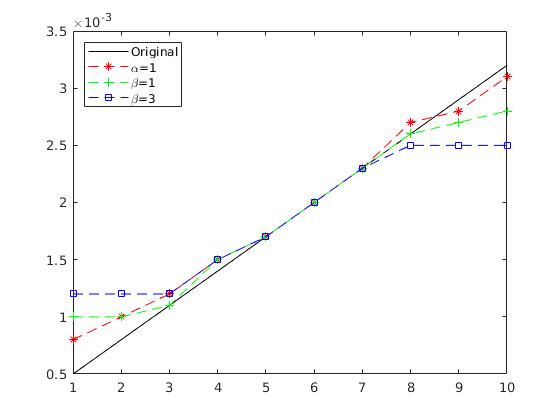}
	\caption{Averaged eigenvalue comparison. \newline}
	\label{sfig:sfig1-f}
\end{subfigure}
\caption{Denoising of a synthetic diffusion tensor image using $p=1.1, s=0.5, n_{\rho} = 3$ and different values of $\alpha$ and $\beta$.}
\label{fig:fig1}
\end{figure}

The \textbf{second denoising example} is shown in \autoref{fig:fig6}. It features one main direction of diffusion. The original image in $W^{s,p}(\Omega,\SPDl_{36}(3))$ is presented in \autoref{sfig:sfig6-a} while the noisy version of it (using $\sigma^2 = 90$) can be seen in \autoref{sfig:sfig6-b}.
Again the size of the ellipsoids in each direction is as before around $10^{-3}$.

Using our regularization method, see \autoref{sfig:sfig6-c}, the noise in all areas is removed while the main direction of diffusion is recognizable. In contrast to this stands the result obtained by using the comparison functional in \autoref{eq:FS}, see \autoref{sfig:sfig6-d}. The main direction is barely visible and noise remains, in particular in regions with tiny ellipsoids. 
Because the size of the ellipsoids is rather small the main contribution in the Sobolev semi-norm regularization is due to the change of size between the larger and smaller diffusion tensors. This leads to the smoothing of the whole image. Furthermore, very tiny ellipsoids barely influence the regularization term which results in the remaining noise.
Compared to that our functional using the log-Euclidean metric results in a completely different behavior. In particular, in this case changes between the small ellipsoids contribute even more than the change of size.

\begin{figure}[!h]
	\centering
	\begin{subfigure}[h]{0.45\linewidth}
		\includegraphics[width=1\linewidth]{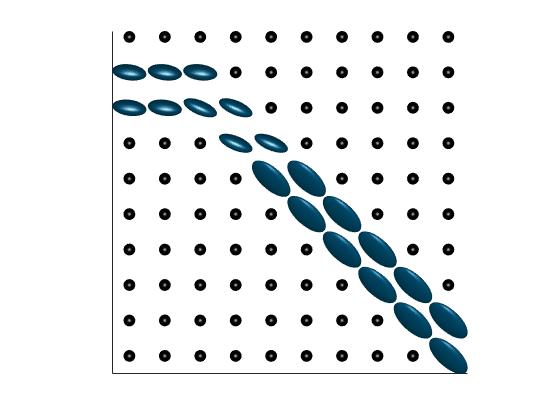}
		\caption{Original data. }
		\label{sfig:sfig6-a}
	\end{subfigure}
	\hspace{1em}
	\begin{subfigure}[h]{0.45\linewidth}
		\includegraphics[width=1\linewidth]{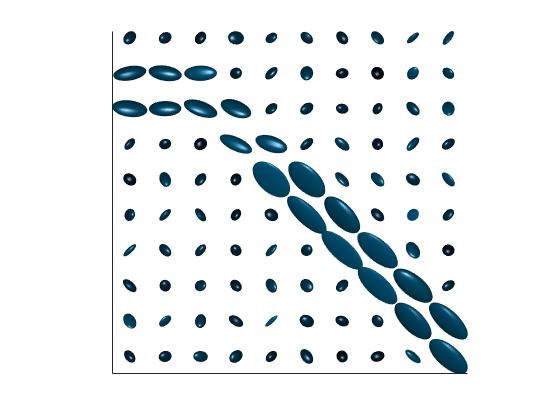}
		\caption{Noisy data using $\sigma^2 = 90$.}
		\label{sfig:sfig6-b}
	\end{subfigure}
	
	\begin{subfigure}[h]{0.45\linewidth}
		\includegraphics[width=1\linewidth]{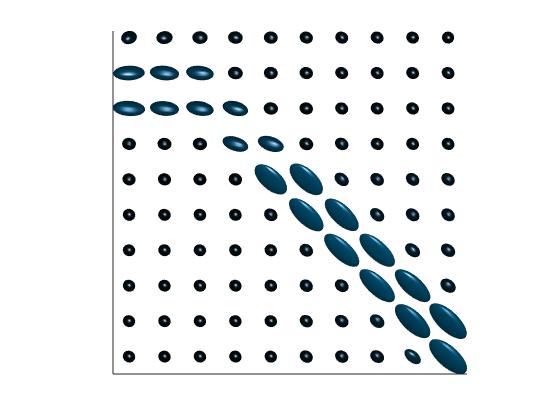}
		\caption{Result with metric double integral regularization with $\alpha=0.3$, $SNR = 7.99$.}
		\label{sfig:sfig6-c}
	\end{subfigure}
	\hspace{1em}
	\begin{subfigure}[h]{0.45\linewidth}
		\includegraphics[width=1\linewidth]{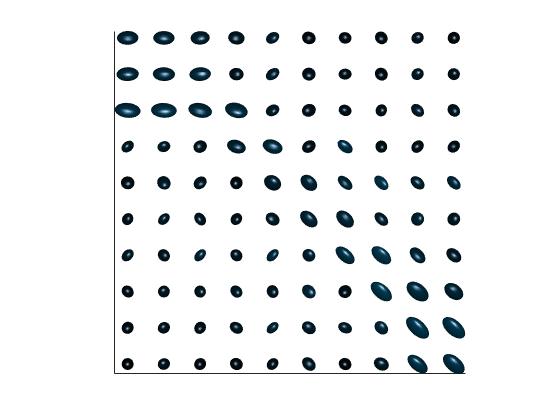}
		\caption{Sobolev semi-norm regularization with $\beta=1$, $SNR = 3.11$.}
		\label{sfig:sfig6-d}
	\end{subfigure}
	\caption{Denoising of a synthetic diffusion tensor image using $p=1.1, s=0.5, n_{\rho} = 2, \alpha = 0.3$ and $\beta = 1$.}
	\label{fig:fig6}
\end{figure}

\subsubsection{Influence of parameters $p$ and $n_\rho$}
\label{subsubsec:influence}
In this section we briefly go into more detail on the influence of the two parameters $p$ and $n_{\rho}$. We again consider the example from the previous section which features one main direction of diffusion. The original image  in $W^{s,p}(\Omega,\SPDl_{36}(3))$ is shown in \autoref{sfig:sfig8-a}. In all images we use $n_{\rho} =9$, i.e. a mollifier that has non-zero support on nine neighbouring pixels. Moreover we chose the same values of $\alpha = 0.3$ and $s=0.5$ as before.\\
In \autoref{sfig:sfig8-b} the parameter $p$ is set to $1.1$, i.e. as in the previous example in \autoref{fig:fig6}. As expected an enlargement of the support of the mollifier leads to a much smoother image in total. In \autoref{sfig:sfig8-c} and \autoref{sfig:sfig8-d} we chose $p=1.001$ and $p=2.1$, respectively. We see that for $p=1.0001$ parts of the main direction of diffusion are still recognizable, i.e. the results stays more close to the original image even when using a mollifier with large support. Increasing the value of $p$ smoothes the whole image, which we also see in \autoref{sfig:sfig8-d}. By an appropriate adaption of the other parameters $s$ and $\alpha$ this effect could possibly be reduced.

\begin{figure}[!h]
	\centering
	\begin{subfigure}[h]{0.45\linewidth}
		\includegraphics[width=1\linewidth]{mainDir/orig.jpg}
		\caption{Original data. }
		\label{sfig:sfig8-a}
	\end{subfigure}
	\hspace{1em}
	\begin{subfigure}[h]{0.45\linewidth}
		\includegraphics[width=1\linewidth]{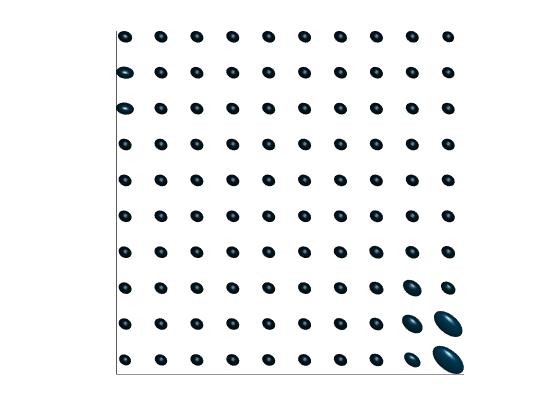}
		\caption{Denoised data using $p=1.1$ and $n_{\rho} = 9$.}
		\label{sfig:sfig8-b}
	\end{subfigure}
	
	\begin{subfigure}[h]{0.45\linewidth}
		\includegraphics[width=1\linewidth]{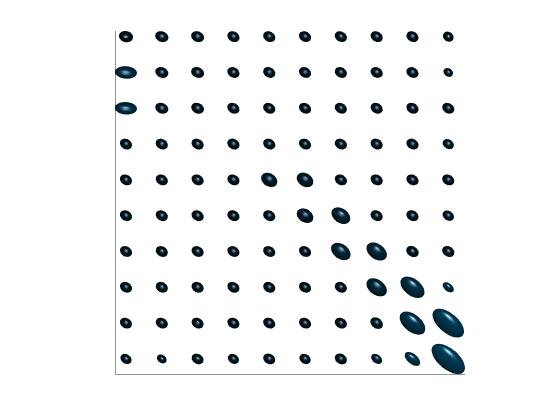}
		\caption{Denoised data using $p=1.0001$ and $n_{\rho} = 9$.}
		\label{sfig:sfig8-c}
	\end{subfigure}
	\hspace{1em}
	\begin{subfigure}[h]{0.45\linewidth}
		\includegraphics[width=1\linewidth]{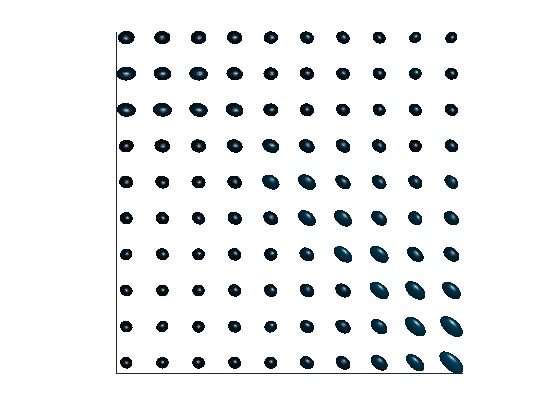}
		\caption{Denoised data using $p=2.1$ and $n_{\rho} = 9$.}
		\label{sfig:sfig8-d}
	\end{subfigure}
	\caption{Influence of different values of $p, s=0.5, \alpha=0.3$ using a mollifier with a broader support with $n_{\rho} = 9$.}
	\label{fig:fig8}
\end{figure}

\subsubsection{Inpainting of synthetic data}
\label{subsubsec:syninpainting}
We now come to two examples of diffusion tensor inpainting for functions in $W^{s,p}(\Omega,\SPDl_{36}(3))$. We thus minimize the functional \autoref{eq:reg}, with $D \neq \emptyset$, which denotes the inpainting domain. 

The \textbf{first example}, where the ground truth is represented in \autoref{sfig:sfig7-a} has one main diffusion direction. 
The noisy image in \autoref{sfig:sfig7-b} is obtained as described in \autoref{ssec:noisyData}  with variance $\sigma^2 = 90$. 
The area $D$ to be inpainted consists of the missing ellipsoids in the noisy data. As input data for our algorithm we use the incomplete noisy data (as shown in \autoref{sfig:sfig7-b}) where we replaced the missing ellipsoids (described by the null matrix $0_n$) by its projection $\Proj(0_n)$, as defined in \autoref{eq:P}.

The result using our metric double integral regularization method can be seen in \autoref{sfig:sfig7-c}. The main diffusion direction is recognizable even though the size of the ellipsoids near the kink is now approximately the same. Noise, which was in particular present in the tiny ellipsoids, is removed because of the use of the log-Euclidean metric in our functional. Small values thus gain a high contribution.
The result using the comparison functional in \autoref{eq:FS} is shown in \autoref{sfig:sfig7-d}. The noise is removed but it is barely possible to recognize the main diffusion direction. The whole image is smoothed. Choosing $\beta$ even smaller the influence of the regularizer tends to zero yielding a result close to the starting image.

\begin{figure}[!h]
	\centering
	\begin{subfigure}[h]{0.45\linewidth}
		\includegraphics[width=1\linewidth]{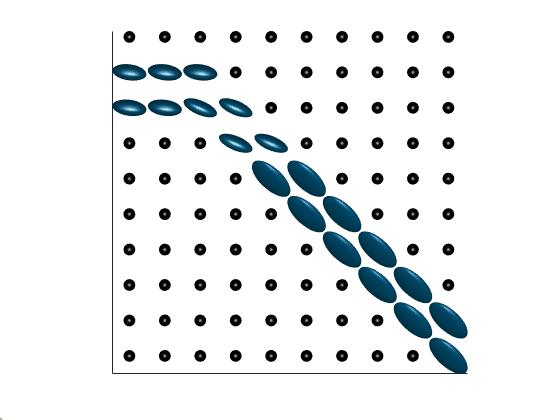}
		\caption{Original data}
		\label{sfig:sfig7-a}
	\end{subfigure}
	\hspace{1em}
	\begin{subfigure}[h]{0.45\linewidth}
		\includegraphics[width=1\linewidth]{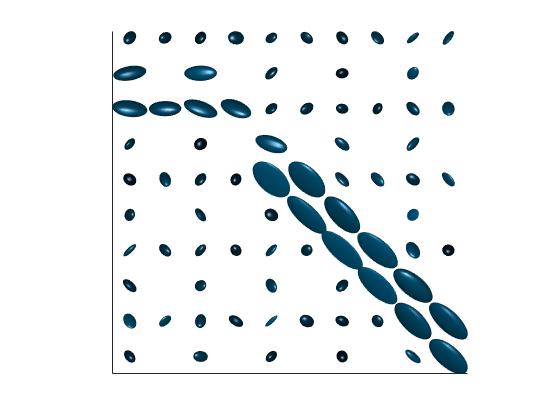}
		\caption{Noisy data to be inpainted. $\sigma^2 = 90$.}
		\label{sfig:sfig7-b}
	\end{subfigure}
	
	\begin{subfigure}[h]{0.45\linewidth}
		\includegraphics[width=1\linewidth]{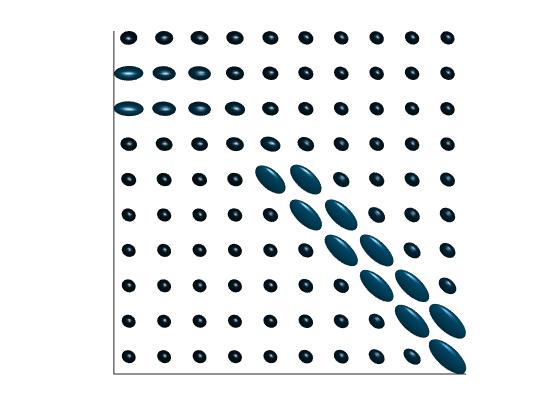}
		\caption{Result with metric double integral regularization with $\alpha=0.5$, $SNR = 4.59$.}
		\label{sfig:sfig7-c}
	\end{subfigure}
	\hspace{1em}
	\begin{subfigure}[h]{0.45\linewidth}
		\includegraphics[width=1\linewidth]{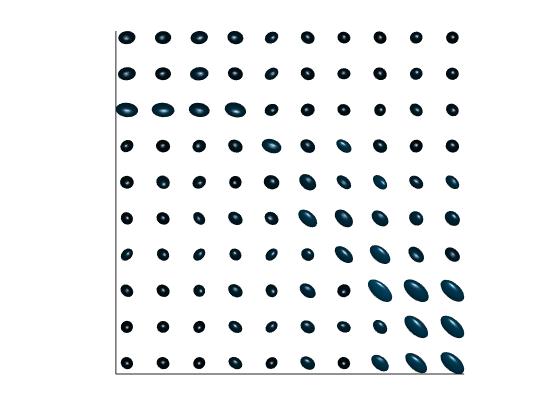}
		\caption{Sobolev semi-norm regularization with $\beta=1$, $SNR = 3.02$.}
		\label{sfig:sfig7-d}
	\end{subfigure}
	\caption{Inpainting of a synthetic diffusion tensor image using $p=1.1, s=0.5, n_{\rho} = 2, \alpha = 0.5$ and $\beta = 1$.}
	\label{fig:fig7}
\end{figure}

As \textbf{second example} we consider the data shown in \autoref{fig:fig3}. The original data is illustrated in \autoref{sfig:sfig3-a}, the noisy one using $\sigma^2 = 40$ in \autoref{sfig:sfig3-b}. This serves as initial data for our minimizing algorithm. The area to be inpainted, $D$, can be seen in \autoref{sfig:sfig3-c}: it consists of the square of missing ellipsoids in the middle.

Using our regularization functional results in \autoref{sfig:sfig3-d}. Using the Sobolev semi-norm regularization with different values of $\beta$ gives \autoref{sfig:sfig3-e} and \autoref{sfig:sfig3-f}. Our result is more balanced concerning noise removal and keeping the inpainted area, in particular the size of the ellipsoids, close to the ground truth data. This is also visible in the value of the $SNR$. When minimizing the comparison functional in \autoref{eq:FS} with a small value of the regularization parameter $\beta$ the size of the ellipsoids is matched well but noise remains. Increasing of $\beta$ leads to a better noise removal with a simultaneous smoothing of the whole image. 

\begin{figure}[!h]
	\centering
	\begin{subfigure}[h]{0.37\linewidth}
		\includegraphics[width=1\linewidth]{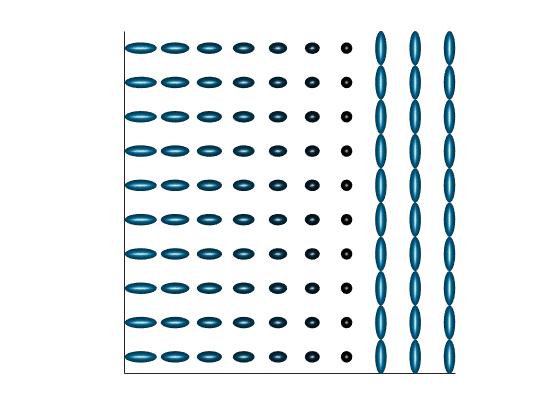}
		\caption{Original data.}
		\label{sfig:sfig3-a}
	\end{subfigure}
	\hspace{1em}
	\begin{subfigure}[h]{0.37\linewidth}
		\includegraphics[width=1\linewidth]{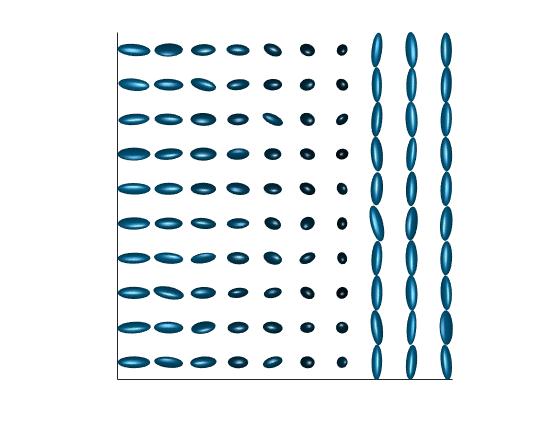}
		\caption{Noisy data. Initial image.}
		\label{sfig:sfig3-b}
	\end{subfigure}
	
	\begin{subfigure}[h]{0.37\linewidth}
		\includegraphics[width=1\linewidth]{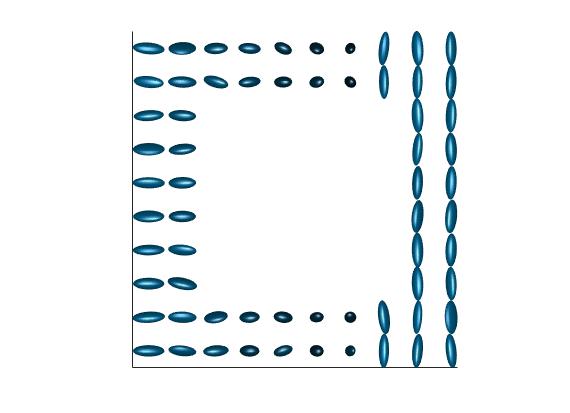}
		\caption{Noisy data to be inpainting.\newline }
		\label{sfig:sfig3-c}
	\end{subfigure}
	\hspace{1em}
	\begin{subfigure}[h]{0.37\linewidth}
		\includegraphics[width=1\linewidth]{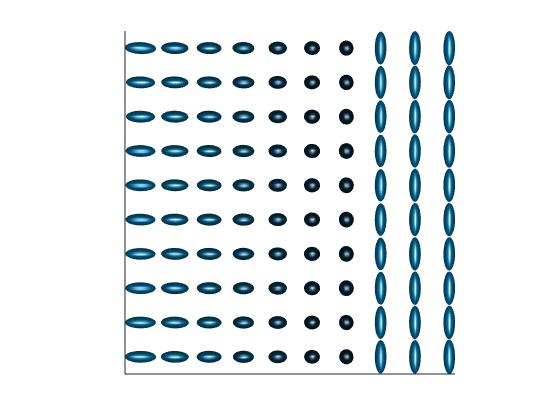}
		\caption{Result with metric double integral regularization with $\alpha=0.5$, $snr = 15.69$.}
		\label{sfig:sfig3-d}
	\end{subfigure}

    \begin{subfigure}[h]{0.37\linewidth}
    	\includegraphics[width=1\linewidth]{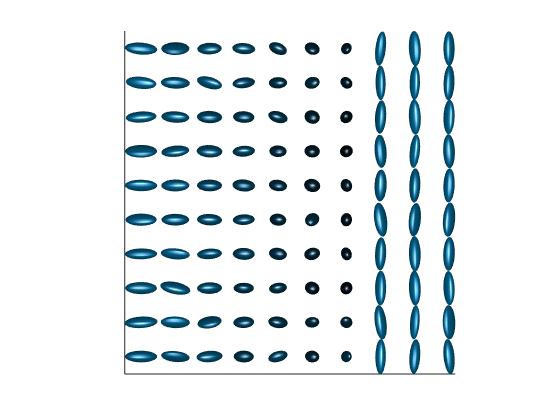}
    	\caption{Sobolev semi-norm regularization with $\beta=0.5$, $SNR = 9.82$.}
    	\label{sfig:sfig3-e}
    \end{subfigure}
    \hspace{1em}
    \begin{subfigure}[h]{0.37\linewidth}
    	\includegraphics[width=1\linewidth]{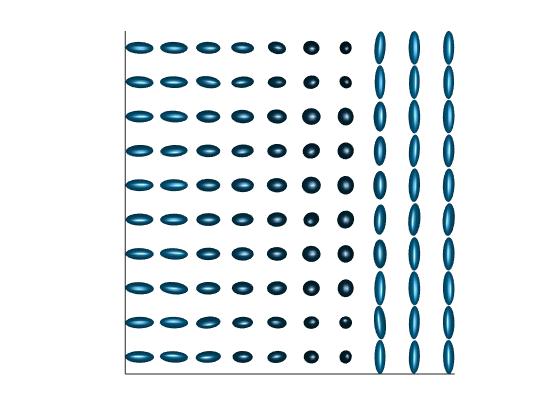}
    	\caption{Sobolev semi norm regularization with $\beta=1$, $SNR = 9.27$.}
    	\label{sfig:sfig3-f}
    \end{subfigure}
	\caption{Inpainting of a synthetic diffusion tensor image using $p=1.1, s=0.5, n_{\rho} = 2, \alpha = 0.5$ and $\beta = 0.5$ and $\beta = 1$, respectively.}
	\label{fig:fig3}
\end{figure}

\subsubsection{Denoising of DTMRI data}
\label{subsubsec:realdenoising}
In this last subsection we present an example for denoising of a real DTMRI image.
The original data are taken from \cite{CabAndBasMai}, which is freely accessible. In this example (parts of) the 39th slice are shown. 
Noise is added with $\sigma^2 = 0.05$. 

In \autoref{sfig:sfig4-c}, \autoref{sfig:sfig4-e} and \autoref{sfig:sfig4-d}, \autoref{sfig:sfig4-f}, respectively, parts of the whole image in \autoref{sfig:sfig4-a} and \autoref{sfig:sfig4-b}, respectively, are shown. The denoised results using our regularization method can be seen in \autoref{sfig:sfig4-g} and \autoref{sfig:sfig4-h}, respectively.
In \autoref{sfig:sfig4-g} we see that the structure and sizes of the ellipsoids are preserved. Nevertheless, noise is still visible in some parts. Increasing the regularization parameter $\alpha$ further leads to more noise removal accompanied by a swelling in particular of those ellipsoids  in the middle of the image which have one eigenvalue close to zero.
In \autoref{sfig:sfig4-h} this effect is visible. Here noise is removed well and the main structures are preserved but there is a swelling of some ellipsoids.

\begin{figure}[!h]
	\centering
	\begin{subfigure}[h]{0.49\linewidth}
		\includegraphics[width=1\linewidth]{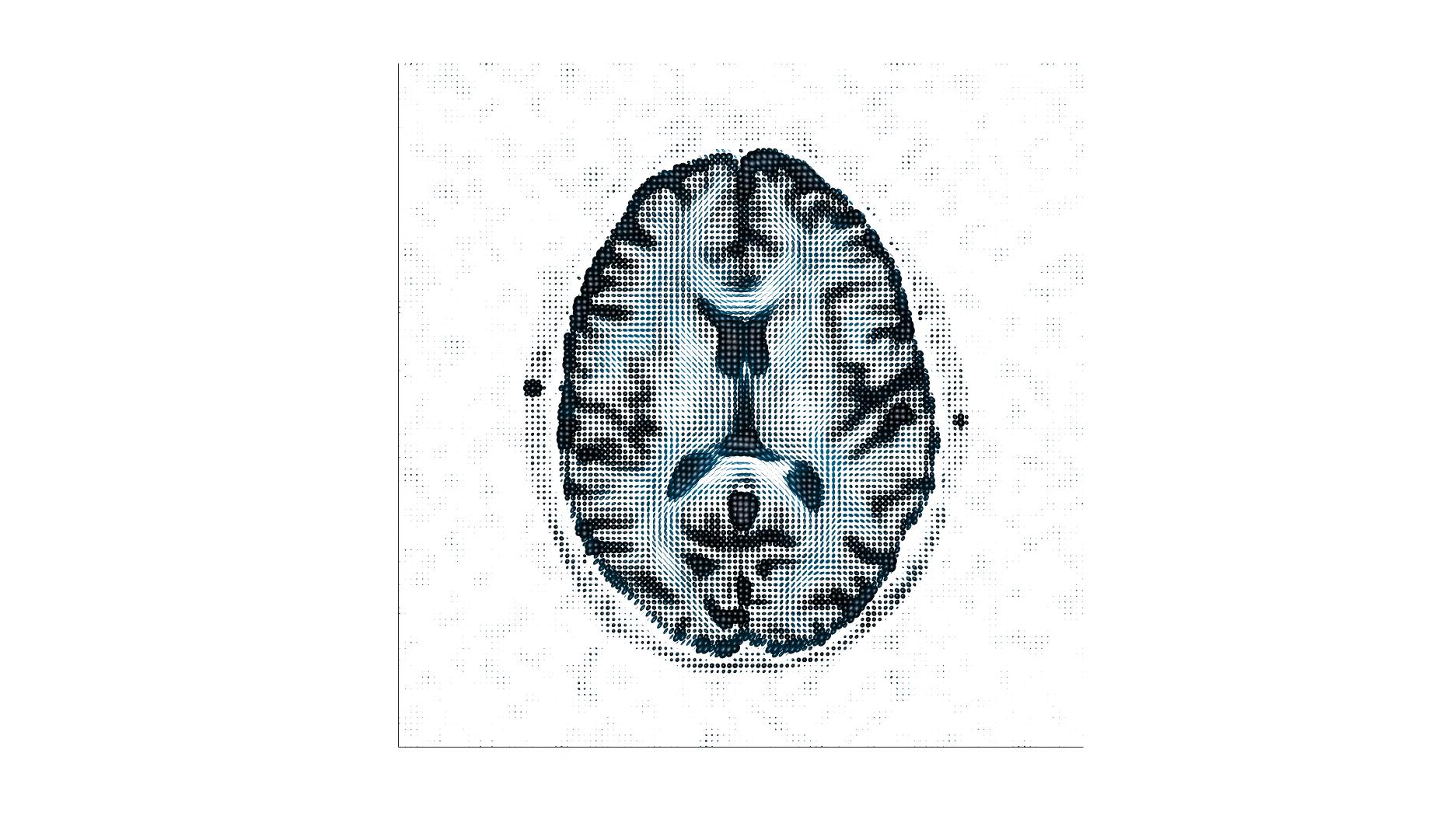}
		\caption{Original data.}
		\label{sfig:sfig4-a}
	\end{subfigure} 
	\begin{subfigure}[h]{0.49\linewidth}
		\includegraphics[width=1\linewidth]{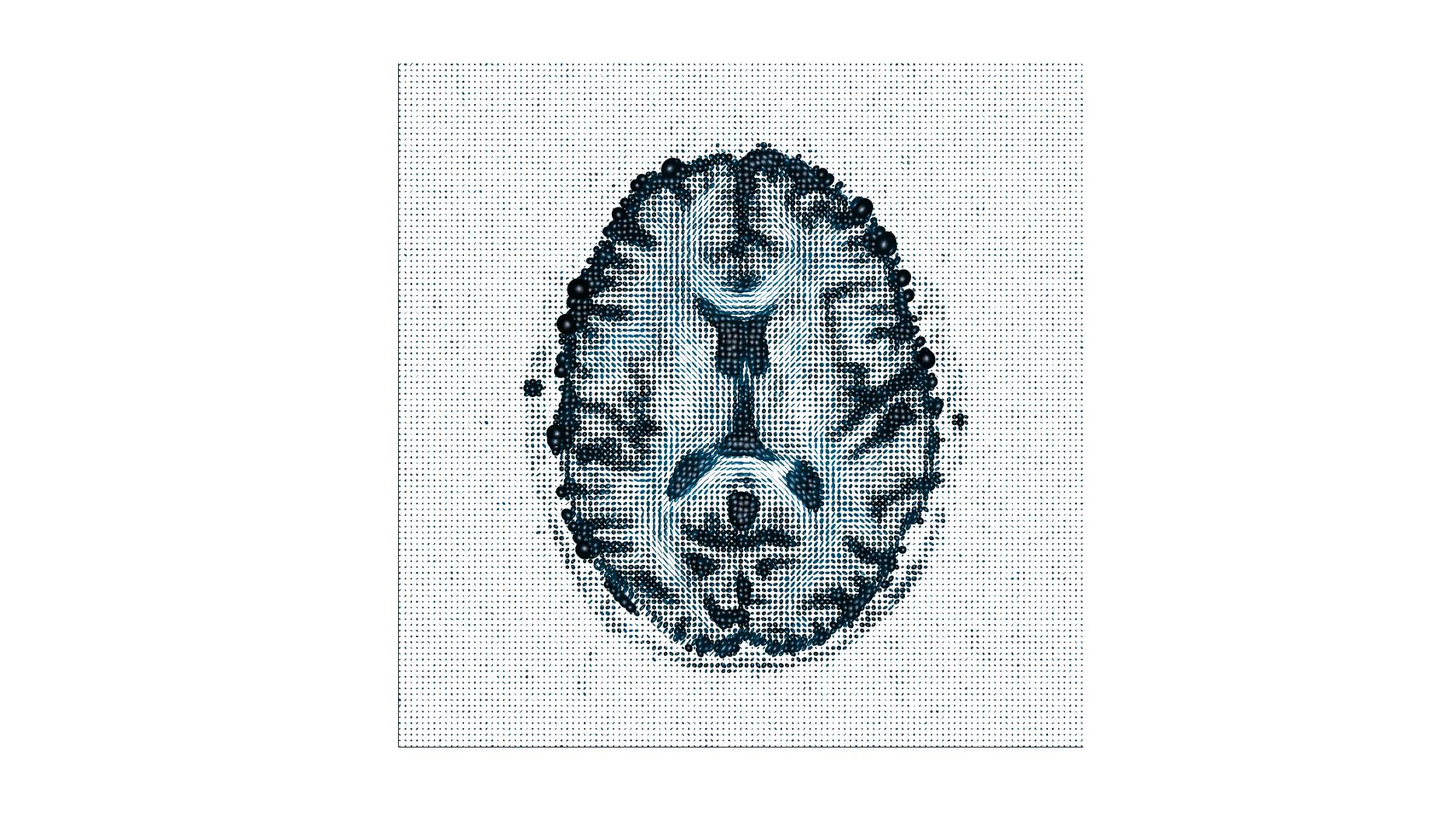}
		\caption{Noisy data.}
		\label{sfig:sfig4-b}
	\end{subfigure}
	
	\begin{subfigure}[h]{0.49\linewidth}
		\includegraphics[width=1\linewidth]{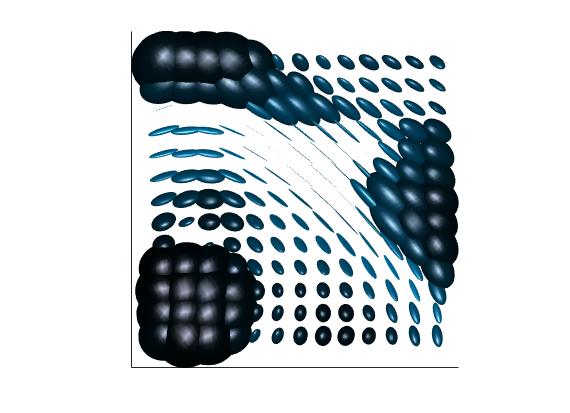}
		\caption{(Part of) original data.}
		\label{sfig:sfig4-c}
	\end{subfigure}
    \begin{subfigure}[h]{0.49\linewidth}
    	\includegraphics[width=1\linewidth]{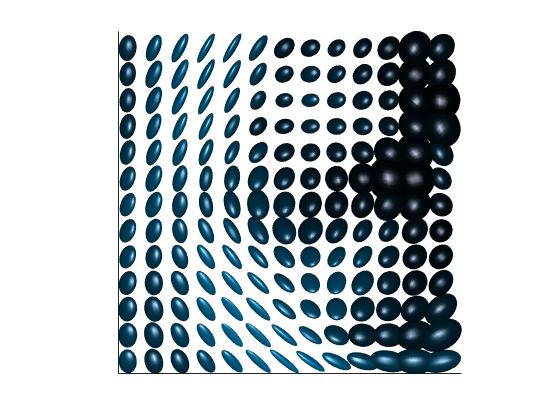}
    	\caption{(Part of) original data.}
    	\label{sfig:sfig4-d}
    \end{subfigure}

	\begin{subfigure}[h]{0.49\linewidth}
		\includegraphics[width=1\linewidth]{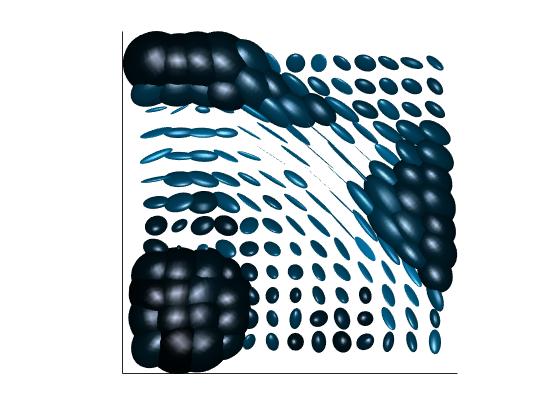}
		\caption{(Part of) noisy data, $\sigma^2 = 0.05, SNR = 7.21$ }
		\label{sfig:sfig4-e}
	\end{subfigure}
    \begin{subfigure}[h]{0.49\linewidth}
    	\includegraphics[width=1\linewidth]{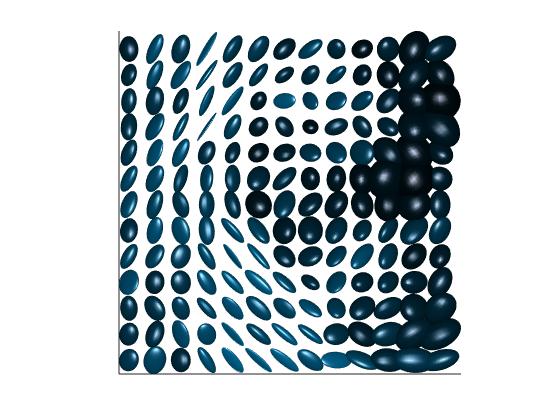}
    	\caption{(Part of) noisy data, $\sigma^2 = 0.05, SNR = 5.85$ }
    	\label{sfig:sfig4-f}
    \end{subfigure}

	\begin{subfigure}[h]{0.49\linewidth}
		\includegraphics[width=1\linewidth]{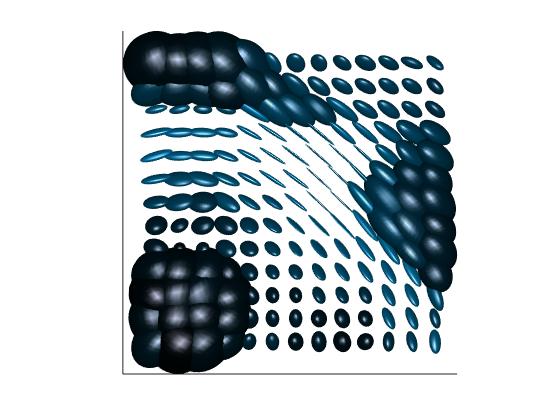}
		\caption{Result with metric double integral regularization with $\alpha=0.7$, $SNR = 8.38$.}
		\label{sfig:sfig4-g}
	\end{subfigure}
    \begin{subfigure}[h]{0.49\linewidth}
    	\includegraphics[width=1\linewidth]{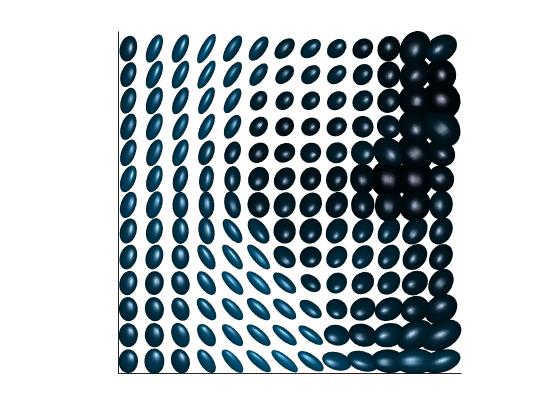}
    	\caption{Result with metric double integral regularization with $\alpha=0.5$, $SNR = 8.64$.}
    	\label{sfig:sfig4-h}
    \end{subfigure}	
	\caption{Denoising of real data taken from \cite{CabAndBasMai} using $p = 1.1, s = 0.1, \alpha = 0.7$ and $\alpha = 0.5$, respectively, and $n_{\rho}=1$.}
	\label{fig:fig4}
\end{figure}

\color{black}
\subsection{Conclusion}
The contribution of this paper is the application of recently developed derivative-free, metric double integral regularization methods for denoising of diffusion tensor imaging data.
The analysis is based on recent work \cite{CiaMelSch19} but completed by a uniqueness result for the minimizer of the regularization functional. In order to derive the analytical 
result we require differential geometric results on sets of positive definite, symmetric matrices. We also demonstrate the effectiveness of the approach by some synthetic and DTMRI data.

\subsection*{Acknowledgements}
MM and OS are supported by the Austrian Science Fund (FWF), with Project 
I3661-N27 (Novel Error Measures and Source Conditions of Regularization Methods 
for Inverse Problems).
Moreover, OS is also by the FWF within the special research initiative SFB F68, 
project F6807-N36 (Tomography with Uncertainties). 

\section*{References}
\renewcommand{\i}{\ii}
\printbibliography[heading=none]


\end{document}